\documentclass[12pt]{article}
\usepackage{amsmath,amsthm,amssymb,euscript,verbatim}
\newtheorem{theorem}{Theorem}[section]

\newtheorem{lemma}{Lemma}[section]
\newtheorem{corollary}{Corollary}[section]
\newtheorem{remark}{Remark}[section]
\newtheorem{example}{Example}[section]
\newtheorem{conjecture}{Conjecture}[section]
\begin{document}
\title
{\bf Classifications of Dupin Hypersurfaces in Lie Sphere Geometry}
\author
{Thomas E. Cecil}
\maketitle

\begin{abstract}
This is a survey of local and global classification results concerning Dupin hypersurfaces in 
$S^n$ (or ${\bf R}^n$) that have been obtained in the context of Lie sphere geometry.  The emphasis is on results that relate Dupin hypersurfaces to isoparametric hypersurfaces in spheres.  Along with these classification results,
many important concepts from Lie sphere geometry, such as curvature spheres, Lie curvatures, and Legendre lifts of submanifolds of $S^n$ (or ${\bf R}^n$), are described in detail. The paper also contains several 
important constructions of Dupin hypersurfaces with certain special properties.
\end{abstract}

\noindent
An oriented hypersurface $f:M^{n-1} \rightarrow S^n$ (or ${\bf R}^n$) is said to be {\em Dupin} if:

\begin{enumerate}
\item[(a)] along each curvature surface, the corresponding principal
curvature is constant.
\end{enumerate}
The hypersurface $M$ is called {\em proper Dupin} if, in addition 
to Condition (a), the following condition is satisfied:

\begin{enumerate}
\item[(b)] the number $g$ of distinct principal curvatures is constant on
$M$.
\end{enumerate}
An important class of proper Dupin hypersurfaces consists of the isoparametric (constant principal curvatures)
hypersurfaces in $S^n$, and those hypersurfaces in ${\bf R}^n$ obtained from isoparametric hypersurfaces in $S^n$ via
stereographic projection.  For example,  a standard product torus $S^1(r) \times S^1(s) \subset S^3$, $r^2+s^2=1$, is an
isoparametric hypersurface with two principal curvatures, and 
the well-known ring cyclides of Dupin \cite{D}, \cite[pp. 148--159]{Cec1} in ${\bf R}^3$ are obtained
from standard product tori via stereographic projection.

In this paper, we will discuss various local and global classification results for proper Dupin hypersurfaces 
in $S^n$ (or ${\bf R}^n$) that have been obtained in the context of Lie sphere geometry.
The emphasis is on results that relate Dupin hypersurfaces to isoparametric hypersurfaces 
in spheres.  

Along with these classification results,
many important concepts from Lie sphere geometry, such as curvature spheres, Lie curvatures, and Legendre lifts of submanifolds of $S^n$, are described in detail.  The paper also contains several important
constructions of Dupin hypersurfaces with certain special properties.
The presentation is based primarily on the 
author's book \cite{Cec1}, and some passages are direct quotations from that book.  

\section{Basic Concepts of Lie Sphere Geometry}
\label{chap:1} 

\subsection{Indefinite space forms and projective space}
\label{sec:1.1}

We begin with some 
preliminary remarks on indefinite scalar product spaces and projective
geometry, along with the corresponding notation.  A {\em scalar product} is a nondegenerate bilinear form on
a real vector space $V$.
We will use the notation ${\bf R}^n_k$ for the vector space $V = {\bf R}^n$ endowed with a scalar
product having signature $(n-k,k)$, and we will primarily consider the cases where $k = 0,1$
or 2.  However, at times we will consider subspaces of ${\bf R}^n_k$
on which the bilinear form is degenerate.

Let $(x,y)$ be the indefinite scalar product of signature $(n,1)$ on the Lorentz space
${\bf R}^{n+1}_1$ defined by
\begin{equation}
\label{eq:1.1.1}
(x,y) = - x_1y_1 + x_2y_2 + \cdots + x_{n+1}y_{n+1},
\end{equation}
where $x = (x_1,\ldots,x_{n+1})$ and $y = (y_1,\ldots,y_{n+1})$.  We will call
this scalar product the {\em Lorentz metric}.
A vector $x$ is said to
be {\em spacelike}, {\em timelike} or {\em lightlike}, respectively,
depending on whether $(x,x)$ is positive, negative or zero.  We
will use this terminology even when we are using a metric of
different signature.
In Lorentz space, the set of all lightlike
vectors, given by the equation,
\begin{equation}
\label{eq:1.1.2}
x_1^2 = x_2^2 + \cdots + x_{n+1}^2,
\end{equation}
forms a cone of revolution, called the 
{\em light cone}.  Timelike vectors are
``inside the cone'' and spacelike vectors are ``outside the cone.''

If $x$ is a nonzero vector in ${\bf R}^{n+1}_1$, let $x^{\perp}$ 
denote the orthogonal complement of $x$ with respect to the Lorentz metric.
If $x$ is timelike, then the metric restricts to a positive definite
form on $x^{\perp}$, and $x^{\perp}$ intersects the light cone only at
the origin.  If $x$ is spacelike, then the metric has signature $(n-1,1)$
on $x^{\perp}$, and $x^{\perp}$ intersects the cone in a cone of
one less dimension.  If $x$ is lightlike, then $x^{\perp}$ is tangent to
the cone along the line through the origin determined by $x$.  The
bilinear form has signature $(n-1,0)$ on this $n$-dimensional plane.

Lie sphere geometry is defined in the context of real projective space ${\bf P}^n$,
so we now recall some concepts from projective geometry.
We define an equivalence relation on ${\bf R}^{n+1} - \{0\}$ by setting
$x \simeq y$ if $x = ty$ for some nonzero real number $t$.  We denote
the equivalence class determined by a vector $x$ by $[x]$.  Projective
space ${\bf P}^n$ is the set of such equivalence classes, and it can
naturally be identified with the space of all lines through the origin
in ${\bf R}^{n+1}$.  The rectangular coordinates $(x_1, \ldots, x_{n+1})$
are called {\em homogeneous coordinates}
of the point $[x] \in {\bf P}^n$, and they
are only determined up to a nonzero scalar multiple.  

The affine space
${\bf R}^n$ can be embedded in ${\bf P}^n$ as the complement of the hyperplane
$(x_1 = 0)$ at infinity by the map $\phi: {\bf R}^n \rightarrow
{\bf P}^n$ given by $\phi (u) = [(1,u)]$.  A scalar product on 
${\bf R}^{n+1}$, such as the Lorentz metric, determines a polar
relationship between points and hyperplanes in ${\bf P}^n$.  We
will also use the notation $x^{\perp}$ to denote the polar hyperplane
of $[x]$ in ${\bf P}^n$, and we will call $[x]$ the {\em pole} of $x^{\perp}$.

If $x$ is a lightlike vector in  ${\bf R}^{n+1} - \{0\}$, then $[x]$ can be represented
by a vector of the form $(1,u)$ for $u \in {\bf R}^n$.  Then the equation
$(x,x) = 0$ for the light cone becomes $u \cdot u = 1$ (Euclidean dot product),
i.e., the equation for the unit sphere in  ${\bf R}^n$.  Hence, the set of 
points in ${\bf P}^n$ determined by lightlike vectors in  ${\bf R}^{n+1} - \{0\}$
is naturally diffeomorphic to the sphere $S^{n-1}$.

\subsection{M\"{o}bius geometry of unoriented spheres}
\label{sec:1.2}

In this section, we study the space of all (unoriented) hyperspheres in Euclidean $n$-dimensional space ${\bf R}^n$ 
and in the unit
sphere $S^n \subset {\bf R}^{n+1}$.  These two spaces of spheres are closely related via stereographic
projection, as we now recall.  We will always assume that $n \geq 2.$  (See \cite[pp. 11--14]{Cec1} for more detail.)

We denote the Euclidean dot product of two vectors $u$ and $v$ in
${\bf R}^n$ by $u \cdot v$.  We first consider stereographic
projection $\sigma : {\bf R}^n \rightarrow S^n - \{P\}$, where
$S^n$ is the unit sphere in ${\bf R}^{n+1}$ given by $y \cdot y = 1$,
and $P = (-1,0,\ldots,0)$ is the south pole 
of $S^n$.  
The well-known formula for $\sigma (u)$ is
\begin{displaymath}
\sigma (u) = \left(\frac{1-u \cdot u}{1+ u \cdot u}, \frac{2u}{1+ u \cdot u}
\right).
\end{displaymath}

We next embed ${\bf R}^{n+1}$ into ${\bf P}^{n+1}$ by the embedding $\phi$
mentioned in the preceding section.  Thus, we have the map 
$\phi \sigma : {\bf R}^n \rightarrow {\bf P}^{n+1}$ given by
\begin{equation}
\label{eq:1.2.1}
\phi \sigma (u) = \left[ \left( 1,\frac{1-u \cdot u}{1+ u \cdot u}, \frac{2u}{1+ u \cdot u}
\right) \right] = \left[ \left( \frac{1+u \cdot u}{2},\frac{1-u \cdot u}{2},u \right) \right].
\end{equation}
Let $(z_1,\ldots,z_{n+2})$ be homogeneous coordinates on ${\bf P}^{n+1}$
and $(\ ,\ )$ the Lorentz metric on the space ${\bf R}^{n+2}_1$.  Then
$\phi \sigma ({\bf R}^n)$ is the set of points in ${\bf P}^{n+1}$
lying on the $n$-sphere $\Sigma$ given by the equation $(z,z) = 0$,
with the exception of the {\em improper point} $[(1,-1,0,\ldots,0)]$
corresponding to the south pole $P$.  We will refer to 
the points in $\Sigma$ other than $[(1,-1,0,\ldots,0)]$ as
{\em proper points}, and we will call $\Sigma$ the 
{\em M\"{o}bius sphere} or {\em M\"{o}bius space}. 

The basic framework for the M\"{o}bius geometry of unoriented spheres
is as follows.  Suppose that $\xi$ is a spacelike vector in ${\bf R}^{n+2}_1$.
Then the polar hyperplane $\xi^{\perp}$ to $[\xi]$ in ${\bf P}^{n+1}$
intersects the sphere $\Sigma$ in an $(n-1)$-sphere $S^{n-1}$.
This sphere $S^{n-1}$ is the image under $\phi \sigma$ of an $(n-1)$-sphere in ${\bf R}^n$,
unless it contains the improper point $[(1,-1,0,\ldots,0)]$, in which case
it is the image under $\phi \sigma$ of a hyperplane in ${\bf R}^n$.  Hence,
we have a bijective correspondence between the set of all spacelike points in 
${\bf P}^{n+1}$ and the set of all hyperspheres and hyperplanes in ${\bf R}^n$.

We now find specific formulas for this correspondence.
Consider the sphere in ${\bf R}^n$ with center $p$ and radius $r>0$
given by the equation
\begin{equation}
\label{eq:1.2.2}
(u-p) \cdot (u-p) = r^2, \quad u \in {\bf R}^n.
\end{equation}
We now translate this into an equation involving the Lorentz metric
and the corresponding polarity relationship on ${\bf P}^{n+1}$.  A direct calculation
shows that equation (\ref{eq:1.2.2}) is equivalent to the equation
\begin{equation}
\label{eq:1.2.3}
(\xi,\phi \sigma (u)) = 0,
\end{equation}
where $\xi$ is the spacelike vector,
\begin{equation}
\label{eq:1.2.4}
\xi = \left( \frac{1+p \cdot p - r^2}{2}, \frac{1- p \cdot p +r^2}{2},p \right),
\end{equation}
and $\phi \sigma (u)$ is given by equation (\ref{eq:1.2.1}).  Thus, the point $u$ is on the sphere
given by equation (\ref{eq:1.2.2}) if and only if $\phi \sigma (u)$ lies on the polar
hyperplane of $[\xi]$.  Note that the first two coordinates
of $\xi$ satisfy $\xi_1 + \xi_2 = 1$, and that
$(\xi,\xi) = r^2$.  Although $\xi$ is only determined
up to a nonzero scalar multiple, we can conclude that 
$\eta_1 + \eta_2$ is not zero for any $\eta \simeq \xi$.

Conversely, given a spacelike point $[z]$ with $z_1 + z_2$ nonzero,
we determine the corresponding sphere in ${\bf R}^n$ as follows.
Let $\xi = z/(z_1 + z_2)$ so that $\xi_1 + \xi_2 = 1$. Then from equation
(\ref{eq:1.2.4}), the center of the corresponding sphere is the point
$p = (\xi_3,\ldots,\xi_{n+2})$, and the radius is the square root of
$(\xi,\xi)$.

We next suppose that $\eta$ is a spacelike vector with $\eta_1 + \eta_2 =0$.
Then
\begin{displaymath}
(\eta, (1,-1,0,\ldots,0)) = 0.
\end{displaymath}
In that case, the improper point $\phi (P)$ lies on the polar hyperplane of $[\eta ]$,
and the point $[\eta ]$ corresponds to a hyperplane in ${\bf R}^n$.
Again we can find an explicit correspondence as follows.  

Consider the hyperplane in
${\bf R}^n$ given by the equation
\begin{equation}
\label{eq:1.2.5}
u \cdot N = h, \quad |N| = 1.
\end{equation}
A direct calculation shows that (\ref{eq:1.2.5}) is equivalent to the equation
\begin{equation}
\label{eq:1.2.6}
(\eta, \phi \sigma (u)) = 0, {\rm where} \ \eta = (h, -h, N).
\end{equation}
So the hyperplane (\ref{eq:1.2.5}) is represented in the polarity relationship
by $[\eta ]$.  Conversely, let $z$ be a spacelike point with
$z_1 + z_2 = 0$.  Then $(z,z) = v \cdot v$, where $v = (z_3,\ldots,z_{n+2})$.
Let $\eta = z / |v|$.  Then $\eta$ has the form (\ref{eq:1.2.6}) and
$[z]$ corresponds to the hyperplane (\ref{eq:1.2.5}).  Thus we have explicit formulas for
the bijective correspondence between the set of spacelike points in ${\bf P}^{n+1}$
and the set of hyperspheres and hyperplanes in ${\bf R}^n$.

The fundamental invariant of M\"{o}bius geometry is the angle.
The study of angles in our setting is quite natural, since orthogonality between
spheres and planes in ${\bf R}^n$ can be expressed in terms of the Lorentz
metric.  Specifically, let $S_1$ and $S_2$ denote the spheres in ${\bf R}^n$ with respective
centers $p_1$ and $p_2$ and respective radii $r_1$ and $r_2$.  By the
Pythagorean Theorem,
the two spheres intersect orthogonally 
if and only if
\begin{equation}
\label{eq:1.2.7}
|p_1 - p_2|^2 = r_1^2 + r_2^2.
\end{equation}

If these spheres correspond by equation (\ref{eq:1.2.4}) to the projective points
$[\xi_1]$ and $[\xi_2]$, respectively, then a calculation shows that equation
(\ref{eq:1.2.7}) is equivalent to the condition
\begin{equation}
\label{eq:1.2.8}
(\xi_1,\xi_2) = 0.
\end{equation}
A hyperplane $\pi$ in ${\bf R}^n$ is orthogonal to a hypersphere $S$ precisely
when $\pi$ passes through the center of $S$.  If $S$ has center $p$ and
radius $r$, and $\pi$ is given by the equation $u \cdot N = h$, then the
condition for orthogonality is just $p \cdot N = h$.  If $S$ corresponds
to $[\xi]$ as in (\ref{eq:1.2.4}) and $\pi$ corresponds to $[\eta]$ as in
(\ref{eq:1.2.6}), then this equation for orthogonality is equivalent to
$(\xi, \eta) = 0$.  Finally, if two planes $\pi_1$ and $\pi_2$ are represented
by $[\eta_1]$ and $[\eta_2]$ as in (\ref{eq:1.2.6}), then the orthogonality
condition $N_1 \cdot N_2 = 0$ is equivalent to the equation
$(\eta_1, \eta_2) = 0$.

A {\em M\"{o}bius transformation} is a projective transformation of ${\bf P}^{n+1}$
which preserves the condition $(\eta, \eta) = 0$. This implies (see, for example,
\cite[p. 26]{Cec1}) that 
a M\"{o}bius transformation also preserves the relationship $(\eta, \xi) = 0$,
and it maps spacelike points to spacelike points.  Thus it preserves
orthogonality (and hence angles) between spheres and planes in ${\bf R}^n$.  
In fact, the group of M\"{o}bius transformations is isomorphic
to $O(n+1,1)/\{\pm I\}$, where $O(n+1,1)$ is the group of orthogonal transformations
of the Lorentz space ${\bf R}^{n+2}_1$ (see, for example, \cite[p. 27]{Cec1}).

Note that a M\"{o}bius transformation takes lightlike vectors to lightlike vectors,
and so it induces a conformal diffeomorphism of the sphere $\Sigma$ onto itself.
It is well known that the group of conformal diffeomorphisms of the sphere
is precisely the M\"{o}bius group.

\subsection{Lie geometry of oriented hyperspheres in ${\bf R}^n$}
\label{sec:1.3}
We now turn to the construction of Lie's \cite{Lie} geometry of oriented spheres in ${\bf R}^n$.
(See \cite[pp. 14--18]{Cec1} for more detail.)
Let $W^{n+1}$ be the set of vectors in ${\bf R}^{n+2}_1$ satisfying $(\zeta, \zeta) = 1.$  This
is a hyperboloid 
of revolution of one sheet in ${\bf R}^{n+2}_1$.  If $\alpha$ is a spacelike
point in ${\bf P}^{n+1}$, then there are precisely two vectors $\pm \zeta$ in $W^{n+1}$ with
$\alpha = [\zeta]$.  These two vectors can be taken to correspond to the two orientations
of the oriented sphere or plane represented by $\alpha$ in the following way.  First, embed ${\bf R}^{n+2}_1$ into 
${\bf P}^{n+2}$ by the embedding
$z \mapsto [(z,1)]$.  If $\zeta \in W^{n+1}$, then
\begin{displaymath}
- \zeta_1^2 + \zeta_2^2 + \cdots + \zeta_{n+2}^2 = 1,
\end{displaymath}
so the point $[(\zeta, 1)]$ in ${\bf P}^{n+2}$ lies on the quadric $Q^{n+1}$ in ${\bf P}^{n+2}$
given in homogeneous coordinates by the equation
\begin{equation}
\label{eq:1.3.1}
\langle x,x \rangle = - x_1^2 + x_2^2 + \cdots + x_{n+2}^2 - x_{n+3}^2 = 0.
\end{equation}
The manifold $Q^{n+1}$ is called the {\em Lie quadric}, and the scalar product
determined by the quadratic form in (\ref{eq:1.3.1}) is called the {\em Lie metric}
or {\em Lie scalar product}.  
We will let $\{e_1,\ldots,e_{n+3}\}$ denote the standard orthonormal
basis for the scalar product space ${\bf R}^{n+3}_2$ with metric $\langle \  ,\  \rangle$.  Here $e_1$
and $e_{n+3}$ are timelike and the rest are spacelike.

We now define a correspondence between the points on $Q^{n+1}$ and the set of oriented hyperspheres,
oriented hyperplanes
and point spheres in ${\bf R}^n \cup \{ \infty\}$.  Suppose that $x$ is any point on the 
quadric with homogeneous coordinate $x_{n+3} \neq 0$.  Then $x$ can be represented by a vector of the form
$(\zeta, 1)$, where the Lorentz scalar product $(\zeta, \zeta) = 1$.  Suppose first that
$\zeta_1 + \zeta_2 \neq 0$.  Then in M\"{o}bius geometry $[\zeta]$ represents a sphere in ${\bf R}^n$.  If as in
equation (\ref{eq:1.2.4}), we represent $[\zeta]$ by a vector of the form
\begin{displaymath}
\xi = \left( \frac{1+p \cdot p - r^2}{2}, \frac{1- p \cdot p +r^2}{2},p \right),
\end{displaymath}
then $(\xi, \xi) = r^2$.  Thus $\zeta$ must be one of the vectors $\pm \xi/r$.  In ${\bf P}^{n+2}$, we have
\begin{displaymath}
[(\zeta, 1)] = [(\pm \xi / r, 1)] = [(\xi, \pm r)].
\end{displaymath}
We can interpret the last coordinate as a signed radius of the sphere with center $p$ and unsigned
radius $r > 0$.  In order to interpret this geometrically, we adopt the convention that a positive signed radius 
corresponds to the orientation of the sphere
determined by the inward field
of unit normals, and a negative signed radius corresponds to the orientation given by the 
outward field of unit
normals.  Hence, the two orientations of the sphere in ${\bf R}^n$ with center $p$ and unsigned radius $r > 0$ are
represented by the two projective points,
\begin{equation}
\label{eq:1.3.2}
\left[ \left( \frac{1+p \cdot p - r^2}{2}, \frac{1- p \cdot p +r^2}{2},p, \pm r \right) \right]
\end{equation}
in $Q^{n+1}$.  Next if $\zeta_1 + \zeta_2 = 0$, then $[\zeta]$ represents a 
hyperplane in ${\bf R}^n$,
as in equation (\ref{eq:1.2.6}).  For $\zeta = (h,-h, N)$, with $|N| = 1$, we have $(\zeta, \zeta) = 1$.  Then the two
projective points on $Q^{n+1}$ induced by $\zeta$ and $- \zeta$ are
\begin{equation}
\label{eq:1.3.3}
[(h, -h, N, \pm 1)].
\end{equation}
These represent the two orientations of the plane with equation $u \cdot N = h$.  We make the
convention that $[(h, -h, N, 1)]$ corresponds to the orientation given by the field of unit
normals $N$, while the orientation given by $-N$ corresponds to the point
$[(h, -h, N, -1)] = [(-h, h, -N, 1)]$.

Thus far we have determined a bijective correspondence between the set of points $x$ in $Q^{n+1}$ with
$x_{n+3} \neq 0$ and the set of all oriented spheres and planes in ${\bf R}^n$.  Suppose now that
$x_{n+3} = 0$, i.e., consider a point $[(z,0)]$, for $z \in {\bf R}^{n+2}_1$.  Then
$\langle x, x \rangle = (z, z) = 0$, and $[z] \in {\bf P}^{n+1}$ is simply a point of the 
M\"{o}bius sphere $\Sigma$.  Thus we have the following bijective correspondence between objects
in Euclidean space and points on the Lie quadric:

\begin{equation}
\label{eq:1.3.4}
\begin{array}{cc}
{\rm {\bf Euclidean}} & {\rm {\bf Lie}} \\
 & \\
\mbox{{\rm points:} }u \in {\bf R}^n & \left[ \left( \frac{1+u \cdot u}{2},\frac{1-u \cdot u}{2},u,0 \right) 
\right]  \\
 & \\
\infty & [(1,-1,0,0)]\\
 & \\
\mbox{{\rm spheres: center} $p$, {\rm signed radius} $r$} & \left[ \left( \frac{1+p \cdot p - r^2}{2}, 
\frac{1- p \cdot p +r^2} {2},p, r \right) \right] \\
 & \\ 
\mbox{{\rm planes:} $u \cdot N = h$, {\rm unit normal} $N$} & [(h,-h,N,1)] 
\end{array}
\end{equation}

\vspace{.25in}

In Lie sphere geometry, points are considered to be spheres of radius zero, or point spheres.
From now on, we will use the term {\em Lie sphere}
or simply ``sphere'' to denote an oriented
sphere, an oriented plane or a point sphere in ${\bf R}^n \cup \{ \infty\}$.  We will refer to the
coordinates on the right side of equation (\ref{eq:1.3.4}) as the {\em Lie coordinates} 
of the corresponding
point, sphere or plane.  In the case of ${\bf R}^2$ and ${\bf R}^3$, respectively, these coordinates
were classically called {\em pentaspherical} and {\em hexaspherical} coordinates (see \cite{Bl}).  

It is useful to have formulas to convert Lie coordinates back into Cartesian equations for the corresponding
Euclidean object.  Suppose first that $[x]$ is a point on the Lie quadric with $x_1 + x_2 \neq 0$.
Then $x = \rho y$, for some $\rho \neq 0$, where $y$ is one of the standard forms on the right side of the table 
(\ref{eq:1.3.4}) above.
From the table, we see that $y_1 + y_2 = 1$, for all proper points and all spheres.  Hence if we divide $x$
by $x_1 + x_2$, the new vector will be in standard form, and we can read off the corresponding Euclidean
object from the table.  In particular, if $x_{n+3} = 0$, then $[x]$ represents the point $u = (u_3,\ldots,u_{n+2})$ where
\begin{equation}
\label{eq:1.3.5}
u_i = x_i / (x_1 + x_2),\quad 3 \leq i \leq n+2.
\end{equation}
If $x_{n+3} \neq 0$, then $[x]$ represents the sphere with center $p = (p_3,\ldots,p_{n+2})$ and signed
radius $r$ given by
\begin{equation}
\label{eq:1.3.6}
p_i = x_i / (x_1 + x_2),\quad 3 \leq i \leq n+2; \quad r = x_{n+3} / (x_1 + x_2).
\end{equation}
Finally, suppose that $x_1 + x_2 = 0$. If $x_{n+3} = 0$, then the equation $\langle x,x \rangle = 0$
forces $x_i$ to be zero for $3 \leq i \leq n+2$.  Thus $[x] = [(1,-1,0,\ldots,0)]$, the
improper point.  If $x_{n+3} \neq 0$, we divide $x$ by $x_{n+3}$ to make the last coordinate 1.
Then if we set $N = (N_3,\ldots,N_{n+2})$ and $h$ according to
\begin{equation}
\label{eq:1.3.7}
N_i = x_i/ x_{n+3},\quad 3 \leq i \leq n+2;\quad h = x_1/x_{n+3},
\end{equation}
the conditions $\langle x,x \rangle = 0$ and $x_1 + x_2 = 0$ force $N$ to have unit length.  Thus
$[x]$ corresponds to the hyperplane $u \cdot N = h$, with unit normal $N$ and $h$ as in equation (\ref{eq:1.3.7}).

\subsection{Oriented hyperspheres in $S^n$}
\label{sec:1.4} 
In some ways it is simpler to use the sphere $S^n$ 
rather than ${\bf R}^n$ as the base space for the
study of M\"{o}bius or Lie sphere geometry.  This avoids the use of stereographic projection 
and the need to refer to an improper point
or to distinguish between spheres and planes.  Furthermore, the correspondence
in the table in equation (\ref{eq:1.3.4}) can be reduced to a single formula (\ref{eq:1.4.4}) below.

As in Section \ref{sec:1.2}, we consider $S^n$ to be the unit sphere in ${\bf R}^{n+1}$, and then embed
${\bf R}^{n+1}$ into ${\bf P}^{n+1}$ by the canonical embedding $\phi$.  Then $\phi (S^n)$ is the
M\"{o}bius sphere $\Sigma$, given by the equation $(z,z) = 0$ in homogeneous coordinates.
First we find the M\"{o}bius equation for the unoriented hypersphere in $S^n$ with center $p \in S^n$
and spherical radius 
$\rho$, for $0 < \rho < \pi$.  This hypersphere is the intersection of $S^n$ with the hyperplane
in ${\bf R}^{n+1}$ given by the equation 
\begin{equation}
\label{eq:1.4.1}
p \cdot y = \cos \rho, \quad 0 < \rho < \pi.
\end{equation}
Let $[z] = \phi(y) = [(1,y)]$.  Then
\begin{displaymath}
p \cdot y = \frac{-(z,(0,p))}{(z,e_1)}.
\end{displaymath}
Thus equation (\ref{eq:1.4.1}) can be rewritten as
\begin{equation}
\label{eq:1.4.2}
(z,(\cos \rho,p)) = 0.
\end{equation}
Therefore, a point $y \in S^n$ is on the hyperplane
determined by equation (\ref{eq:1.4.1}) if and only if $\phi (y)$
lies on the polar hyperplane in ${\bf P}^{n+1}$ of the point
\begin{equation}
\label{eq:1.4.3}
[\xi] = [(\cos \rho,p)].
\end{equation}
To obtain the two oriented spheres determined by equation (\ref{eq:1.4.1}) note that
\begin{displaymath}
(\xi, \xi ) = - \cos^2 \rho + 1 = \sin^2 \rho.
\end{displaymath}
Noting that $\sin \rho \neq 0$, we let $\zeta = \pm \xi / \sin \rho$.  Then the point $[(\zeta, 1)]$
is on the quadric $Q^{n+1}$, and
\begin{displaymath}
[(\zeta, 1)] = [(\xi, \pm \sin \rho)] = [(\cos \rho, p, \pm \sin \rho)].
\end{displaymath}
We can incorporate the sign of the last coordinate into the radius and thereby arrange that the 
oriented sphere $S$ with 
signed radius $\rho \neq 0, - \pi < \rho < \pi$, and center $p$ corresponds to a point in $Q^{n+1}$
as follows:
\begin{equation}
\label{eq:1.4.4}
S \longleftrightarrow [(\cos \rho, p, \sin \rho)].
\end{equation}
The formula still makes sense if the radius $\rho = 0$, in which case it yields the point sphere
$[(1, p, 0)]$.  This one formula (\ref{eq:1.4.4}) plays the role of all the formulas given in equation (\ref{eq:1.3.4})
in the preceding section for the Euclidean case.

As in the Euclidean case, the orientation of a sphere $S$ in $S^n$ is
determined by a choice of unit normal field to $S$ in $S^n$.  Geometrically, we take the positive radius in 
(\ref{eq:1.4.4}) to correspond to the field of unit normals which are tangent vectors to geodesics from $-p$ to $p$.
Each oriented sphere can be considered in two ways, with center $p$ and signed radius $\rho, - \pi < \rho < \pi$,
or with center $-p$ and the appropriate signed radius $\rho \pm \pi$.

Given a point $[x]$ in the quadric $Q^{n+1}$, we now determine the corresponding hypersphere in $S^n$.
Multiplying by $-1$, if necessary, we may assume that the first homogeneous coordinate $x_1$ of $x$
satisfies $x_1 \geq 0$.  If $x_1 > 0$, then we see from (\ref{eq:1.4.4}) that the center $p$ and signed radius
$\rho, - \pi/2 < \rho < \pi/2$, satisfy
\begin{equation}
\label{eq:1.4.5}
\tan \rho = x_{n+3}/x_1, \quad p = (x_2, \ldots, x_{n+2})/ (x_1^2 + x_{n+3}^2)^{1/2}.
\end{equation}
If $x_1 = 0$, then $x_{n+3} \neq 0$, so we can divide by $x_{n+3}$ to obtain a point of the form
$(0,p,1)$.  This corresponds to the oriented hypersphere with center $p$ and signed radius 
$\pi / 2$, which
is a great sphere in $S^n$.  

One can construct the space of oriented hyperspheres in hyperbolic space $H^n$ in a similar way
(see \cite[p. 18]{Cec1}).

\subsection{Oriented contact of spheres}
\label{sec:1.5}
In M\"{o}bius geometry, the principal geometric quantity is the angle. In Lie sphere geometry,
the corresponding central concept is that of oriented contact of spheres.  
(See \cite[pp. 19--23]{Cec1} for more detail and proofs of the results mentioned in this section.)

Two oriented spheres
$S_1$ and $S_2$ in ${\bf R}^n$ are in {\em oriented contact}
if they are tangent to each other, and they
have the same orientation at the point of contact. 

If $p_1$ and $p_2$ are the respective centers of $S_1$ and $S_2$, and $r_1$ and $r_2$ are their respective
signed radii, then the analytic condition for oriented contact is
\begin{equation}
\label{eq:1.5.1}
|p_1 - p_2| = |r_1 - r_2|.
\end{equation}
An oriented sphere $S$ with center $p$ and signed radius $r$ is in oriented contact with an oriented
hyperplane $\pi$ with unit normal $N$ and equation $u \cdot N = h$, if $\pi$ is tangent to $S$ and their 
orientations agree at the point of contact.  Analytically, this is just the equation
\begin{equation}
\label{eq:1.5.2}
p \cdot N = r + h.
\end{equation}
Two oriented planes $\pi_1$ and $\pi_2$ are in oriented contact if their unit normals $N_1$ and $N_2$ are
the same.  Two such planes can be thought of as two oriented spheres in oriented contact at the improper point.

A proper point $u$ in ${\bf R}^n$ is in oriented contact with a sphere or plane if it lies on the sphere or plane.  Finally,
the improper point is in oriented contact with each plane, since it lies on each plane.

Suppose that $S_1$ and $S_2$ are two Lie spheres which are represented in the standard form given in 
equation (\ref{eq:1.3.4}) by $[k_1]$ and $[k_2]$.  One can check directly that in all cases, the analytic
condition for oriented contact is equivalent to the equation
\begin{equation}
\label{eq:1.5.3}
\langle k_1, k_2 \rangle = 0.
\end{equation}

It follows from the linear algebra of indefinite scalar product spaces
that the Lie quadric contains projective lines but 
no linear subspaces of ${\bf P}^{n+2}$ of higher dimension (see \cite[p. 21]{Cec1}).
The set of oriented spheres in ${\bf R}^n$ corresponding
to the points on a line on $Q^{n+1}$ forms a so-called {\em parabolic pencil} 
of spheres. 

The key result in establishing the relationship between the points on a line in $Q^{n+1}$ and the corresponding 
parabolic pencil of spheres in ${\bf R}^n$ is the following (see \cite[p. 22]{Cec1}).
\begin{theorem}
\label{thm:1.5.4} 
{\rm (a)} The line in ${\bf P}^{n+2}$ determined by two points $[k_1]$ and $[k_2]$ of $Q^{n+1}$
lies on $Q^{n+1}$ if and only if the the spheres corresponding to $[k_1]$ and $[k_2]$ are in
oriented contact, i.e., $\langle k_1, k_2 \rangle = 0.$\\
{\rm (b)} If the line $[k_1,k_2]$ lies on $Q^{n+1}$, then the parabolic pencil of spheres in ${\bf R}^n$ corresponding
to points on $[k_1,k_2]$ is precisely the set of all spheres in oriented contact with both $[k_1]$ and $[k_2]$.
\end{theorem}

Given any timelike point 
$[z]$ in ${\bf P}^{n+2}$,
the scalar product $\langle \ ,\  \rangle$ has signature $(n+1,1)$ on $z^\perp$. Hence, $z^\perp$ intersects
$Q^{n+1}$ in a M\"{o}bius space.  
We now show that any line on the quadric intersects such a M\"{o}bius space
at exactly one point.
\begin{corollary}
\label{cor:1.5.5} 
Let $[z]$ be a timelike point in ${\bf P}^{n+2}$ and $\ell$ a line that lies on $Q^{n+1}$.  Then $\ell$
intersects $z^\perp$ at exactly one point.
\end{corollary}
\begin{proof}
Any line in projective space intersects a hyperplane in at least one point.  We simply must show that 
$\ell$ is not contained in $z^\perp$.  This follows from the fact that $\langle \ ,\  \rangle$
has signature $(n+1,1)$ on $z^\perp$, and therefore $z^\perp$ cannot contain the 2-dimensional lightlike
vector space that projects to $\ell$.
\end{proof}

As a consequence, we obtain the following corollary.
\begin{corollary}
\label{cor:1.5.6} 
Every parabolic pencil
contains exactly one point sphere. Furthermore, if the point sphere is a proper point,
then the pencil contains exactly one plane.
\end{corollary}
\begin{proof}
The point spheres are precisely the points of intersection of  $Q^{n+1}$ with $e_{n+3}^\perp$
Thus each parabolic pencil contains exactly one point sphere by Corollary \ref{cor:1.5.5}. 
To prove the second claim, note that the oriented hyperplanes in ${\bf R}^n$
correspond to the points in the intersection of $Q^{n+1}$ with $(e_1 - e_2)^\perp$.  The line $\ell$ on
the quadric corresponding to the given parabolic pencil intersects this hyperplane at exactly one point
unless $\ell$ is contained in  $(e_1 - e_2)^\perp$ .  But $\ell$ is contained in $(e_1 - e_2)^\perp$ if and only if
the improper point $[e_1 - e_2]$ is in $\ell^\perp$.  This implies that the point 
$[e_1 - e_2]$ is on $\ell$.  Otherwise, the 2-dimensional linear lightlike subspace of ${\bf P}^{n+2}$ 
spanned by $[e_1 - e_2]$ and $\ell$ lies on $Q^{n+1}$, which is impossible.
Hence, if the point sphere of the pencil is not the improper point, then the
pencil contains exactly one hyperplane.
\end{proof}

By Corollary \ref{cor:1.5.6} and Theorem \ref{thm:1.5.4}, we see that if the point sphere in a parabolic
pencil is a proper point $p$ in ${\bf R}^n$, then the pencil consists precisely of all spheres in
oriented contact with a certain oriented plane $\pi$ at $p$.  Thus, one can identify the parabolic pencil
with the point $(p,N)$ in the unit tangent bundle 
to ${\bf R}^n$, where $N$ is the unit normal to the
oriented plane $\pi$.  If the point sphere of the pencil is the improper point, then the pencil must
consist entirely of planes.  
Since these planes are all in oriented contact, they all have the same
unit normal $N$.  Thus the pencil can be identified with the point $(\infty, N)$ in the unit tangent
bundle to ${\bf R}^n \cup \{ \infty \} = S^n$.

It is also useful to have this correspondence between parabolic pencils and elements of the unit tangent bundle 
$T_1S^n$ expressed in terms of the spherical metric 
on $S^n$.  Suppose that $\ell$ is a line on the
quadric.  From Corollary \ref{cor:1.5.5} and equation (\ref{eq:1.4.4}), we see that $\ell$ intersects both
$e_1^\perp$ and $e_{n+3}^\perp$ at exactly one point.  So the corresponding 
parabolic pencil contains exactly
one point sphere and one great sphere, represented respectively by the points,
\begin{displaymath}
[k_1] = [(1,p,0)], \quad [k_2] = [(0,\xi,1)].
\end{displaymath}
The fact that $\langle k_1, k_2 \rangle = 0$ is equivalent to the condition $p \cdot \xi = 0$, i.e., $\xi$
is tangent to $S^n$ at $p$.  Hence the parabolic pencil of spheres corresponding to $\ell$ can
be identified with the point $(p, \xi)$ in $T_1S^n$.  The points on the line $\ell$ can be parametrized as
\begin{displaymath}
[K_t] = [\cos t \ k_1 + \sin t \ k_2] = [(\cos t, \cos t \ p + \sin t \ \xi, \sin t)].
\end{displaymath}
From equation (\ref{eq:1.4.4}), we see that $[K_t]$ corresponds to the sphere in $S^n$ with center
\begin{equation}
\label{eq:1.5.4}
p_t = \cos t \ p + \sin t \ \xi,
\end{equation}  
and signed radius $t$.  
These are precisely the spheres through $p$ in oriented contact with the great sphere
corresponding to $[k_2]$.  Their centers lie along the geodesic in $S^n$ with initial point $p$ and initial velocity
vector $\xi$.

\subsection{Lie sphere transformations}
\label{sec:2.1}
A {\em Lie sphere transformation} 
is a projective transformation of ${\bf P}^{n+2}$ which takes $Q^{n+1}$
to itself.  In terms of the geometry of ${\bf R}^n$, a Lie sphere transformation maps 
Lie spheres to Lie spheres. 
Furthermore, since it
is projective, a Lie sphere transformation maps lines on $Q^{n+1}$ to lines on $Q^{n+1}$.  Thus, it preserves
oriented contact of spheres in ${\bf R}^n$ or $S^n$.  

Pinkall \cite{P4} proved
the so-called 
``Fundamental Theorem of Lie sphere geometry,'' 
which states that any line preserving diffeomorphism 
of $Q^{n+1}$ is the restriction to $Q^{n+1}$
of a projective transformation. Thus, a transformation of the space of oriented spheres which preserves
oriented contact must be a Lie sphere transformation.  We will not give the proof here, but refer the reader
to Pinkall's paper \cite[p. 431]{P4} or the book \cite[p. 28]{Cec1}. 

Recall that a linear transformation $A \in GL(n+1)$ induces a 
projective transformation $P(A)$ on
${\bf P}^n$ defined by $P(A) [x] = [Ax]$.  The map $P$ is a homomorphism of $GL(n+1)$ onto the group
$PGL(n)$ of projective transformations 
of ${\bf P}^n$.  It is well known (see, for example, Samuel \cite[p. 6]{Sam})
that the kernel of $P$ is the group of all nonzero scalar multiples of the identity transformation $I$.

The group $G$ of Lie sphere transformations
is isomorphic to the group $O(n+1,2)/\{\pm I\}$, where $O(n+1,2)$ is the group of 
orthogonal transformations of ${\bf R}^{n+3}_2$.  This follows
immediately from the following theorem.
Here we let $\langle \ ,\  \rangle$ denote the scalar product on ${\bf R}^n_k$ (see \cite[p. 26]{Cec1} for a proof).

\begin{theorem}
\label{thm:2.1.1} 
Let $A$ be a nonsingular linear transformation on the 
indefinite scalar product space ${\bf R}^n_k, 1 \leq k \leq n-1$,
such that $A$ takes lightlike vectors to lightlike vectors.\\
{\rm (a)} Then there is a nonzero constant $\lambda$ such that $\langle Av, Aw \rangle = \lambda \langle v,w \rangle$
for all $v,w$ in ${\bf R}^n_k$.\\
{\rm (b)} Furthermore, if $k \neq n-k$, then $\lambda > 0$.
\end{theorem}

\begin{remark}
\label{rem:2.1.2}
{\em In the case $k = n-k$, conclusion (b) does not necessarily hold.  For example, suppose that 
$\{v_1,\ldots,v_k,w_1,\ldots,w_k\}$ is an orthonormal basis for ${\bf R}^{2k}_k$ with
$v_1,\ldots,v_k$ timelike and $w_1,\ldots,w_k$ spacelike.  Then
the linear map $T$ defined by 
$Tv_i = w_i, Tw_i = v_i$, for $1 \leq i \leq k$, preserves lightlike vectors, but the corresponding $\lambda$ equals $-1$.}
\end{remark}

From Theorem~\ref{thm:2.1.1} we immediately obtain the following corollary.

\begin{corollary}
\label{cor:2.1.3} 
{\rm (a)} The group $G$ of Lie sphere transformations is isomorphic to $O(n+1,2)/\{\pm I\}$.\\
{\rm (b)} The group $H$ of M\"{o}bius transformations is isomorphic to $O(n+1,1)/\{\pm I\}$.
\end{corollary}
\begin{proof}
(a) Suppose $\alpha = P(A)$ is a Lie sphere transformation.  By Theorem~\ref{thm:2.1.1}, we have
$\langle Av,Aw \rangle = \lambda \langle v,w \rangle$ for all $v,w$ in ${\bf R}^{n+3}_2$, where $\lambda$ is a positive
constant. Set $B$ equal to $A/\sqrt{\lambda}$.  Then $B$ is in $O(n+1,2)$ and $\alpha = P(B)$.  Thus, every Lie
sphere transformation can be represented by an orthogonal transformation.  Conversely, if $B \in O(n+1,2)$, then $P(B)$ is
clearly a Lie sphere transformation.  Now let $\Psi: O(n+1,2) \rightarrow G$ be the restriction of the homomorphism $P$
to $O(n+1,2)$.  Then $\Psi$ is surjective with kernel equal to the intersection of the kernel of $P$ with
$O(n+1,2)$, i.e., kernel $\Psi = \{\pm I\}$.\\
(b) This follows from Theorem~\ref{thm:2.1.1} in the same manner as (a) with the Lorentz metric being used instead
of the Lie metric.
\end{proof}

\begin{remark}
\label{rem:2.1.4}
On M\"{o}bius transformations in Lie sphere geometry.\\
{\em A M\"{o}bius transformation
is a transformation on the space of unoriented spheres, i.e., the space of projective
classes of spacelike 
vectors in ${\bf R}^{n+2}_1$.  Hence, each M\"{o}bius transformation naturally induces two
Lie sphere transformations on the space $Q^{n+1}$ of oriented spheres.  Specifically, if $A$ is in $O(n+1,1)$, then
we can extend $A$ to a transformation $B$ in $O(n+1,2)$ by setting $B = A$ on ${\bf R}^{n+2}_1$ and
$B(e_{n+3}) = e_{n+3}$.  In terms of matrix representation with respect to the standard orthonormal basis,
$B$ has the form
\begin{equation}
\label{eq:2.1.4}
B = \left[ \begin{array}{cc}A&0\\0&1\end{array}\right].
\end{equation}
Note that while $A$ and $-A$ induce the same M\"{o}bius transformation, the Lie transformation $P(B)$
is not the same as the Lie transformation $P(C)$ induced by the matrix
\begin{displaymath}
C = \left[ \begin{array}{cc}-A&0\\0&1\end{array}\right] \simeq \left[ \begin{array}{cc}A&0\\0&-1\end{array}\right],
\end{displaymath}
where $\simeq$ denotes equivalence as projective transformations.  Hence, the M\"{o}bius transformation $P(A) = P(-A)$
induces two Lie transformations, $P(B)$ and $P(C)$.  Note that $P(B) = \Gamma P(C)$, where $\Gamma$
is the {\em change of orientation transformation} represented in matrix form by
\begin{displaymath}
\Gamma = \left[ \begin{array}{cc}I&0\\0&-1\end{array}\right] \simeq \left[ \begin{array}{cc}-I&0\\0&1\end{array}\right].
\end{displaymath}
From equation (\ref{eq:1.3.4}), 
we see that $\Gamma$ has the effect of changing the orientation of
every oriented sphere or plane.}  
\end{remark}

\section{Submanifolds in Lie Sphere Geometry}
\label{chap:3}

In this section, we develop the framework necessary to study submanifolds 
of real space forms within the context of Lie sphere geometry. (See  \cite[pp. 51--64]{Cec1}) for more detail.)
For convenience, we will deal primarily with submanifolds of $S^n$, and thereby avoid the complications induced by the improper point associated to  ${\bf R}^n$.  This also makes it easier to treat the important class of isoparametric
hypersurfaces in spheres.  However, we will also consider submanifolds of ${\bf R}^n$, when it is more
convenient to do so.

\subsection{Legendre submanifolds}
\label{sec:3.1}

The manifold $\Lambda^{2n-1}$ of 
lines on the Lie quadric $Q^{n+1}$ has a 
contact structure, i.e., a
globally defined 1-form $\omega$ such that $\omega \wedge (d\omega)^{n-1} \neq 0$ on $\Lambda^{2n-1}$. 
The condition $\omega = 0$ defines 
a codimension one distribution $D$ on $\Lambda^{2n-1}$ that has integral submanifolds of dimension $n-1$, but none
of higher dimension (see, for example, \cite[p. 57]{Cec1}).  
Such an integral submanifold of maximal dimension is an immersion $\lambda:M^{n-1} \rightarrow \Lambda^{2n-1}$ satisfying the condition $\lambda^* \omega = 0$.  It is called a {\em Legendre submanifold}  \cite{Cec1} or a {\em Lie geometric hypersurface} \cite{P1}, \cite{P4}.  As we shall see, these Legendre submanifolds are natural generalizations of oriented hypersurfaces in $S^n$ or ${\bf R}^n$.  

Pinkall \cite{P4} (see also \cite[pp. 59--60]{Cec1}) showed that a smooth map 
 $\lambda:M^{n-1} \rightarrow \Lambda^{2n-1}$ 
 with $\lambda = [k_1, k_2]$, for smooth maps $k_1$ and
$k_2$ from $M^{n-1}$ into ${\bf R}^{n+3}_2$, determines 
a Legendre submanifold if and only
if $k_1$ and $k_2$ satisfy the following conditions, which we call {\em Pinkall's Conditions} (1)--(3)\label{PC}.\\

\noindent 
$(1)$ Scalar product conditions: 
For each $x \in M^{n-1}$, the vectors $k_1(x)$ and $k_2(x)$ are linearly independent
and 
\begin{displaymath}
\langle k_1,k_1 \rangle = 0, \quad \langle k_2, k_2 \rangle = 0, \quad \langle k_1, k_2 \rangle = 0.
\end{displaymath}
$(2)$ Immersion condition: 
There is no nonzero tangent vector $X$ at any point $x \in M^{n-1}$ such that
$dk_1(X)$ and $dk_2(X)$ are both in {\em Span} $\{k_1(x),k_2(x)\}$.\\

\noindent
$(3)$ Contact condition: $\langle dk_1(X), k_2 (x)\rangle = 0$ for all $X \in T_xM^{n-1}$, for all  $x \in M^{n-1}$.\\

\noindent
These conditions are preserved
under a reparametrization $\lambda = [\widetilde{k}_1, \widetilde{k}_2]$, where 
$\widetilde{k}_1 = \alpha  k_1 + \beta k_2$ and $\widetilde{k}_2 = \gamma k_1 + \delta k_2$, for smooth functions
$\alpha, \beta, \gamma, \delta$ on $M^{n-1}$ with $\alpha \delta - \beta \gamma \neq 0.$ 

Condition $(1)$ means that the line $[k_1, k_2]$ lies on $Q^{n+1}$ for all $x \in M^{n-1}$, and so $\lambda$ is a map into 
$\Lambda^{2n-1}$.  Condition $(2)$ means that the map $\lambda$
is an immersion. Condition $(3)$ is equivalent to the contact condition
$\lambda^*\omega = 0$.

An oriented hypersurface  $f:M^{n-1} \rightarrow S^n$ with field of unit normals
$\xi:M^{n-1} \rightarrow S^n$ naturally induces a Legendre submanifold defined by the two functions,  
\begin{equation}
\label{eq:3.3.1}
k_1(x) = (1, f(x), 0), \quad k_2(x) = (0,\xi (x), 1).
\end{equation} 
For each $x \in M^{n-1}$, $[k_1(x)]$ is the point sphere on the line $\lambda (x)$, and $[k_2(x)]$ is the great sphere
on $\lambda (x)$.
It is easy to check that the
pair $\{k_1, k_2\}$ satisfies the conditions (1)--(3) as follows.

Condition (1) is immediate since both
$f$ and $\xi$ are maps into $S^n$, and $\xi(x)$ is tangent to $S^n$ at $f(x)$ for each $x$ in $M^{n-1}$.
Condition (2) is satisfied since
\begin{displaymath}
dk_1(X)= (0,df(X),0),
\end{displaymath} 
for any vector $X \in T_xM^{n-1}$.  Since $f$ is an immersion, $df(X)\neq 0$ for a nonzero vector $X$, and thus 
$dk_1(X)$ is not in Span $\{k_1(x),k_2(x)\}$.  Finally, condition (3) is satisfied since for any vector $X \in T_xM^{n-1}$,
\begin{displaymath}
\langle dk_1(X), k_2(x) \rangle = df(X) \cdot \xi(x) = 0,
\end{displaymath} 
because $\xi$ is a field of unit normals to $f$.  So in this case, the contact condition $(3)$ just amounts to the fact that 
$\xi$ is a field of unit normals to $f$.
The map $\lambda$ is called the {\em Legendre lift}  of the oriented hypersurface $f$
with field of unit normals
$\xi$.

Next, we handle the case of a submanifold $\phi: V \rightarrow S^n$ of codimension  greater than one.  Let
$B^{n-1}$ be the unit normal bundle
of the submanifold $\phi$.  Then $B^{n-1}$ can be considered to be the
submanifold of $V \times S^n$ given by
\begin{displaymath}
B^{n-1} = \{ (x,\xi)| \phi(x) \cdot \xi = 0, \ d\phi(X) \cdot \xi = 0,\  \mbox{\rm for all }X \in T_xV\}.
\end{displaymath} 
The {\em Legendre lift}\label{Legendre-lift} of $\phi(V)$ 
is the map $\lambda: B^{n-1} \rightarrow \Lambda^{2n-1}$ defined by
\begin{equation}
\label{eq:3.3.2}
\lambda(x,\xi) = [k_1(x,\xi), k_2(x,\xi)],
\end{equation} 
where
\begin{equation}
\label{eq:3.3.3}
k_1(x,\xi) = (1, \phi(x), 0), \quad k_2(x,\xi) = (0, \xi, 1).
\end{equation} 
Geometrically, $\lambda(x,\xi)$ is the line on the quadric $Q^{n+1}$ corresponding to the 
parabolic pencil
of spheres in $S^n$ in oriented contact at the contact element $(\phi(x),\xi) \in T_1S^n$.  
 In this case, the point sphere map $[k_1]$
has constant rank equal to the dimension of $V$.

We now return to the case of a general Legendre submanifold $\lambda:M^{n-1} \rightarrow \Lambda^{2n-1}$.
 For each $x \in M^{n-1}$,  $\lambda(x)$ is a line on the 
quadric $Q^{n+1}$.  This line contains exactly one
point $[k_1(x)]$ corresponding to a point sphere in $S^n$ and one point $[k_2(x)]$ corresponding to a
great sphere in $S^n$.  The map $[k_1]$ from $M^{n-1}$ to $Q^{n+1}$ is called the 
{\em M\"{o}bius projection} or {\em point sphere map} 
of $\lambda$, and likewise, the map $[k_2]$ is called the 
{\em great sphere map}.

The homogeneous coordinates
of these points with respect to the standard basis are given by
\begin{equation}
\label{eq:3.2.1}
k_1(x) = (1, f(x), 0), \quad k_2(x) = (0,\xi (x), 1),
\end{equation}
where $f$ and $\xi$ are both smooth maps from $M^{n-1}$ to $S^n$ defined by formula (\ref{eq:3.2.1}).  The map
$f$ is called the {\em spherical projection} 
of $\lambda$, and $\xi$ is called the 
{\em spherical field of unit normals}.  
The maps $f$ and $\xi$ depend on the choice of orthonormal basis $\{e_1,\ldots,e_{n+3}\}$
for ${\bf R}^{n+3}_2$.  
For a general Legendre submanifold, the spherical projection $f$ is not necessarily an immersion, nor does it 
necessarily have constant rank. (See \cite[pp. 59--60]{Cec1} for more detail.)

If the range of the point sphere map $[k_1]$ does not contain the improper point $[(1,-1,0,\ldots,0)]$, then
$\lambda$ also determines a {\em Euclidean projection}, 
\begin{displaymath}
F:M^{n-1} \rightarrow {\bf R}^n,
\end{displaymath}
and a {\em Euclidean field of unit normals},
\begin{displaymath}
\eta:M^{n-1} \rightarrow S^{n-1} \subset {\bf R}^n.
\end{displaymath}
These are defined by the equation 
$\lambda = [Z_1, Z_2]$, where
\begin{equation}
\label{eq:3.3.11}
Z_1 = (1+F\cdot F, 1 - F \cdot F, 2F,0)/2, \quad Z_2 = (F \cdot \eta, - (F \cdot \eta), \eta, 1).
\end{equation} 
According to (\ref{eq:1.3.4}), $[Z_1(x)]$ corresponds to the unique point sphere in the parabolic pencil 
determined by $\lambda(x)$,
and $[Z_2(x)]$ corresponds to the unique plane in this pencil.  As in the spherical case, the smooth maps
$F$ and $\eta$ need not have constant rank.  (See \cite[pp. 63--64]{Cec1} for more detail.)

\subsection{Curvature spheres and Dupin hypersurfaces}
\label{sec:3.4}
To motivate the definition of a curvature sphere 
we consider the case of an oriented hypersurface
$f:M^{n-1} \rightarrow S^n$ with field of unit normals $\xi:M^{n-1} \rightarrow S^n$.  The 
{\em shape operator}
of $f$ at a point $x \in M^{n-1}$ is the symmetric linear transformation $A:T_xM^{n-1} \rightarrow T_xM^{n-1}$
defined by the equation
\begin{equation}
\label{eq:3.4.1}
df(AX) = - d\xi(X), \quad X \in T_xM^{n-1}.
\end{equation} 
Often we consider $f$ to be an embedding and suppress the mention of $f$.  Then we identify the tangent vector $X$ with
$df(X)$.  

The eigenvalues of $A$ are the 
{\em principal curvatures}, and the corresponding eigenvectors are
the {\em principal vectors}.  
We next recall the notion of a focal point of an immersion.  For each
real number $t$, define a map
\begin{displaymath}
f_t:M^{n-1} \rightarrow S^n,
\end{displaymath}
by
\begin{equation}
\label{eq:3.4.2}
f_t = \cos t \ f + \sin t \ \xi.
\end{equation} 
For each $x \in M^{n-1}$, the point $f_t(x)$ lies an oriented distance 
$t$ along the normal geodesic
to $f(M^{n-1})$ at $f(x)$.  A point $p = f_t(x)$ is called a 
{\em focal point} {\em of multiplicity} $m>0$
{\em of} $f$ {\em at} $x$ if the nullity of $df_t$ is equal to $m$ at $x$.    

The location of focal points is determined by the
principal curvatures.  Specifically, if $X \in T_xM^{n-1}$, then by equation (\ref{eq:3.4.1}) we have
\begin{equation}
\label{eq:3.4.3}
df_t(X) = \cos t \ df(X) + \sin t \ d\xi(X) = df(\cos t \ X - \sin t \ AX).
\end{equation} 
Thus, $df_t(X)$ equals zero for $X\neq0$ if and only if $\cot t$ is a principal curvature of $f$ at $x$,
and $X$ is a corresponding principal vector.  Hence, $p = f_t(x)$ is a focal point of $f$ at $x$ of
multiplicity $m$ if and only if $\cot t$ is a principal curvature of multiplicity $m$ at $x$.  Note
that each principal curvature 
\begin{displaymath}
\kappa = \cot t, \quad 0<t<\pi,
\end{displaymath}
produces two distinct antipodal focal points
on the normal geodesic with parameter values $t$ and $t+\pi$.  

The oriented hypersphere centered at a focal point
$p$ and in oriented contact with $f(M^{n-1})$ at $f(x)$ is called a 
{\em curvature sphere} of $f$ at $x$.  The
two antipodal focal points determined by $\kappa$ are the two centers of the corresponding curvature sphere.
Thus, the correspondence between principal curvatures and curvature spheres is bijective.  The multiplicity
of the curvature sphere is by definition equal to the multiplicity of the corresponding principal curvature.

We now consider these ideas as they apply to the 
Legendre lift of an oriented hypersurface $f$ with field of unit normals
$\xi$.  As in equation (\ref{eq:3.3.1}), we have $\lambda = [k_1, k_2]$, where
\begin{equation}
\label{eq:3.4.4}
k_1 = (1, f, 0), \quad k_2 = (0, \xi, 1).
\end{equation} 
For each $x \in M^{n-1}$, the points on the line $\lambda(x)$ can be parametrized as
\begin{equation}
\label{eq:3.4.5}
[K_t(x)] = [\cos t \ k_1(x) + \sin t \ k_2(x)] = [(\cos t,f_t(x),\sin t)],
\end{equation} 
where $f_t$ is given in equation (\ref{eq:3.4.2}).  By equation (\ref{eq:1.4.4}), the point
$[K_t(x)]$ in $Q^{n+1}$ corresponds to the oriented sphere in $S^n$ with center $f_t(x)$ and signed 
radius $t$.
This sphere is in oriented contact with the oriented hypersurface $f(M^{n-1})$ at $f(x)$.  Given a tangent vector
$X \in T_xM^{n-1}$, we have
\begin{equation}
\label{eq:3.4.6}
dK_t(X)= (0,df_t(X),0).
\end{equation} 
Thus, $dK_t(X)= (0,0,0)$ if and only if $df_t(X) = 0$, i.e., $p = f_t(x)$ is a focal point of $f$ at $x$.  Hence, we have
shown the following.
\begin{lemma}
\label{lem:3.4.1} 
The point $[K_t(x)]$ in $Q^{n+1}$ corresponds to a curvature sphere of the hypersurface $f$ at $x$ if and only if
$dK_t(X)= (0,0,0)$ for some nonzero vector $X \in T_xM^{n-1}$.
\end{lemma}

This characterization of curvature spheres depends on the parametrization of $\lambda$ given by
$\{k_1, k_2\}$ as in equation (\ref{eq:3.4.4}), 
and it has only been defined in the case where the spherical projection $f$ is an immersion. 
It is easy to show (see \cite[p. 66]{Cec1}) that the following is a generalization of the definition
of a curvature sphere that is valid  for an arbitrary parametrization of an arbitrary Legendre submanifold.

Let $\lambda: M^{n-1} \rightarrow \Lambda^{2n-1}$ be a Legendre submanifold parametrized by the pair $\{Z_1, Z_2\}$,
satisfying Pinkall's Conditions (1)--(3), see page \pageref{PC}.  
Let $x \in M^{n-1}$ and $r,s \in {\bf R}$ with $(r,s) \neq (0,0)$.  The sphere,
\begin{displaymath}
[K] = [r Z_1(x) + s Z_2(x)],
\end{displaymath}
is called a {\em curvature sphere} 
of $\lambda$ at $x$ if there exists a nonzero vector $X$ in $T_xM^{n-1}$ such that
\begin{equation}
\label{eq:3.4.7}
r \ dZ_1(X) + s \ dZ_2 (X) \in \mbox{\rm Span } \{Z_1(x),Z_2(x)\}.
\end{equation} 
The vector $X$ is called a {\em principal vector} 
corresponding to the curvature sphere $[K]$.  

From equation (\ref{eq:3.4.7}), it is clear that the set of principal vectors corresponding to a given curvature
sphere $[K]$ at $x$ is a subspace of $T_xM^{n-1}$.  This set is called the 
{\em principal space} corresponding
to the curvature sphere $[K]$.  Its dimension is the 
{\em multiplicity} of $[K]$.

\begin{remark}
\label{rem:3.4.2}
{\em The definition of curvature sphere can be developed in the context of Lie sphere geometry without any
reference to submanifolds of $S^n$ (see Cecil--Chern 
\cite{CC1} for details).  In that case, one begins with
a Legendre submanifold $\lambda: M^{n-1} \rightarrow \Lambda^{2n-1}$ and considers a curve $\gamma (t)$ lying
in $M^{n-1}$.  The set of points in $Q^{n+1}$ lying on the set of lines $\lambda (\gamma(t))$ forms a
ruled surface\index{ruled surface} 
in $Q^{n+1}$.  One then considers conditions for this ruled surface to be developable.  
This
leads to a system of linear equations whose roots determine the curvature spheres at each point along the curve.}
\end{remark}

We next want to show that the notion of a curvature sphere is invariant under Lie sphere transformations.
Let $\lambda: M^{n-1} \rightarrow \Lambda^{2n-1}$ be a Legendre submanifold parametrized by $\lambda = [Z_1, Z_2]$.
Suppose $\beta = P(B)$ is the Lie sphere transformation induced by an 
orthogonal transformation $B$ in the group $O(n+1,2)$.
Since $B$ is orthogonal, it is easy to check that the maps, $W_1 = BZ_1$, $W_2 = BZ_2$, satisfy Pinkall's
Conditions (1)--(3) for a Legendre submanifold. We will denote the Legendre submanifold defined by 
$\{W_1,W_2\}$ by 
\begin{displaymath}
\beta \lambda: M^{n-1} \rightarrow \Lambda^{2n-1}.
\end{displaymath}
The Legendre submanifolds $\lambda$ 
and $\beta \lambda$ are said to be {\em Lie equivalent}.  In terms of Euclidean geometry, 
suppose that $V$ and
$W$ are two immersed submanifolds of $S^n$ (or ${\bf R}^n$).  We say that $V$ and $W$ are 
{\em Lie equivalent} if their Legendre lifts are Lie equivalent.

Consider $\lambda$ and $\beta$ as above, so that $\lambda = [Z_1, Z_2]$ and $\beta \lambda = [W_1, W_2]$.
Note that for a tangent vector $X \in T_xM^{n-1}$ and for real numbers $(r,s) \neq (0,0)$, we have
\begin{equation}
\label{eq:3.4.8}
r \ dW_1(X) + s \ dW_2(X) = B (r \ dZ_1(X) + s \ dZ_2(X)),
\end{equation} 
since $B$ is linear.  Thus, we see that
\begin{displaymath}
r \ dW_1(X) + s \ dW_2(X) \in \mbox{\rm Span } \{W_1(x),W_2(x)\}
\end{displaymath} 
if and only if
\begin{displaymath}
r \ dZ_1(X) + s \ dZ_2(X) \in \mbox{\rm Span } \{Z_1(x),Z_2(x)\}.
\end{displaymath} 
This immediately implies the following theorem.
\begin{theorem}
\label{thm:3.4.3} 
Let $\lambda: M^{n-1} \rightarrow \Lambda^{2n-1}$ be a Legendre submanifold and $\beta$ a Lie sphere transformation.
The point $[K]$ on the line $\lambda(x)$ is a curvature sphere of $\lambda$ at $x$ if and only if the
point $\beta [K]$ is a curvature sphere of the Legendre submanifold $\beta \lambda$ at $x$.  Furthermore,
the principal spaces corresponding to $[K]$ and $\beta [K]$ are identical. 
\end{theorem}

An important special case is when the Lie sphere transformation is a 
spherical parallel transformation
$P_t$ defined by,
\begin{eqnarray}
\label{eq:3.4.9}
P_t e_1 & = & \cos t \ e_1 + \sin t \ e_{n+3}, \nonumber \\
P_t e_{n+3} & = & - \sin t \ e_1 + \cos t \ e_{n+3},\\
P_t e_i & = & e_i, \quad 2 \leq i \leq n+2. \nonumber
\end{eqnarray}
The transformation $P_t$ has the effect of adding $t$ to the signed radius 
of each sphere in $S^n$ while keeping the center
fixed (see \cite[pp. 48--49]{Cec1}).

Suppose that $\lambda: M^{n-1} \rightarrow \Lambda^{2n-1}$ is a Legendre submanifold parametrized by the 
point sphere and great sphere maps $\{k_1, k_2\}$, as in equation (\ref{eq:3.4.4}).  
Then $P_t \lambda = [W_1, W_2]$,
where 
\begin{equation}
\label{eq:3.4.10}
W_1 = P_t k_1 = (\cos t,f,\sin t), \quad W_2 = P_t k_2 = (- \sin t, \xi, \cos t).
\end{equation} 
Note that $W_1$ and $W_2$ are not the point sphere and great sphere maps for $P_t \lambda$.  Solving for the
point sphere map $Z_1$ and the great sphere map $Z_2$ of $P_t \lambda$, we find
\begin{eqnarray}
\label{eq:3.4.11}
Z_1 & = & \cos t \ W_1 - \sin t \ W_2 = (1, \cos t \ f - \sin t \ \xi, 0),  \\
Z_2 & = &  \sin t \ W_1 + \cos t \ W_2 = (0, \sin t \ f + \cos t \ \xi, 1).\nonumber
\end{eqnarray}
From this, we see that $P_t \lambda$ has spherical projection and spherical unit normal field given, respectively, by
\begin{eqnarray}
\label{eq:3.4.12}
f_{-t} & = & \cos t \ f - \sin t \ \xi = \cos (-t)f + \sin (-t) \xi,  \\
\xi_{-t} & = &  \sin t \ f + \cos t \ \xi = - \sin (-t) f + \cos (-t) \xi.\nonumber
\end{eqnarray}
The minus sign occurs because $P_t$ takes a sphere with center $f_{-t}(x)$ and radius $-t$ to the point sphere
$f_{-t}(x)$.  We call $P_t \lambda$ a {\em parallel submanifold}
of $\lambda$.  

Formula (\ref{eq:3.4.12}) shows
the close correspondence between these parallel submanifolds and the 
{\em parallel hypersurfaces} $f_t$ to $f$.
In the case where the spherical projection $f$ is an immersed hypersurface,
a parallel map $f_t$ has a singularity at $x \in M^{n-1}$
if and only if $p = f_t (x)$ is a focal point of $M^{n-1}$ at $x$. So for each $x \in M^{n-1}$, there are at most
$n-1$ values of $t \in [0,\pi)$ for which $f_t$ fails to be an immersion. 

The following theorem, due to Pinkall 
\cite[p. 428]{P4} (see also \cite[pp. 70--72]{Cec1}), 
shows that this is also true in the case where the spherical projection $f$ is not an immersion.
This theorem is clear if the original spherical projection $f$ is an immersion, but it requires proof
if $f$ has singularities. We omit the proof here, however, and refer the reader to
\cite[p. 428]{P4} or \cite[pp. 70--72]{Cec1}. 

\begin{theorem}
\label{thm:3.4.4} 
Let $\lambda: M^{n-1} \rightarrow \Lambda^{2n-1}$ be a Legendre submanifold with spherical projection $f$ and spherical
unit normal field $\xi$.  Then for each $x \in M^{n-1}$, the parallel map, 
\begin{displaymath}
f_t  =  \cos t \ f + \sin t \ \xi,
\end{displaymath}
fails to be an immersion at $x$ for at most $n-1$ values of $t \in [0,\pi)$.
\end{theorem}

Here $[0,\pi)$ is the appropriate interval, because of the fact mentioned earlier that each principal
curvature of an immersion produces two distinct antipodal focal points in the interval $[0,2\pi)$.  We next
state some important consequences of this theorem that are obtained by passing to a parallel submanifold,
if necessary, and then applying well-known results concerning immersed hypersurfaces in $S^n$. 

\begin{corollary}
\label{cor:3.4.5} 
Let $\lambda: M^{n-1} \rightarrow \Lambda^{2n-1}$ be a Legendre submanifold. Then: \\
{\em (a)} at each point $x \in M^{n-1}$, there are at most $n-1$ distinct curvature spheres $K_1, \ldots, K_g$,\\
{\em (b)} the principal vectors corresponding to a curvature sphere $K_i$ form a subspace $T_i$ of the tangent space
$T_xM^{n-1}$,\\
{\em (c)} the tangent space $T_xM^{n-1} = T_1 \oplus \cdots \oplus T_g$,\\
{\em (d)} if the dimension of a given $T_i$ is constant on an open subset $U$ of $M^{n-1}$, then the
principal distribution $T_i$ is integrable on $U$,\\
{\em (e)} if $\dim T_i = m > 1$ on an open subset $U$ of $M^{n-1}$, then the curvature sphere map $K_i$ is constant
along the leaves of the principal foliation $T_i$.
\end{corollary}

\begin{proof}
In the case where the spherical projection $f$ of $\lambda$ is an immersion, the corollary follows
from known results concerning hypersurfaces in $S^n$ and the correspondence between the curvature spheres of
$\lambda$ and the principal curvatures of $f$.  Specifically, (a)--(c) follow from elementary linear algebra
applied to the (symmetric) shape operator $A$ of the immersion $f$.  As to (d) and (e), 
Ryan \cite[p. 371]{Ryan1}
showed that the principal curvature functions on an immersed hypersurface are continuous. 
Nomizu \cite{Nom2} then
showed that any continuous principal curvature function $\kappa_i$ which has constant multiplicity on an open subset
$U$ in $M^{n-1}$ is smooth on $U$, as is its corresponding principal distribution (see also, 
Singley \cite{Sin}).  If the
multiplicity $m_i$ of $\kappa_i$ equals one on $U$, then $T_i$ is integrable by the theory of ordinary differential 
equations.  If $m_i > 1$, then the integrability of $T_i$, and the fact that $\kappa_i$ is constant along the leaves
of $T_i$ are consequences of Codazzi's equation
(Ryan \cite{Ryan1}, see also 
Cecil--Ryan \cite[p. 24]{CR8}
and Reckziegel \cite{Reck2}).

Note that (a)--(c) are pointwise statements, while (d)--(e) hold on an open set $U$ if they can be shown to
hold in a neighborhood of each point of $U$.  Now let $x$ be an arbitrary point of $M^{n-1}$.  If the spherical
projection $f$ is not an immersion at $x$, then by Theorem~\ref{thm:3.4.4}, we can find a 
parallel transformation
$P_{-t}$ such that the spherical projection $f_t$ of the Legendre submanifold $P_{-t} \lambda$ is an immersion
at $x$, and hence on a neighborhood of $x$.  By Theorem~\ref{thm:3.4.3}, the corollary also holds for $\lambda$
in this neighborhood of $x$.  Since $x$ is an arbitrary point, the corollary is proved.
\end{proof}

Let $\lambda: M^{n-1} \rightarrow \Lambda^{2n-1}$ be an arbitrary Legendre submanifold. A connected submanifold
$S$ of $M^{n-1}$ is called a 
{\em curvature surface} if at each $x \in S$, the tangent space $T_xS$ is
equal to some principal space $T_i$.  For example, if $\dim T_i$ is constant on an open subset $U$ of $M^{n-1}$,
then each leaf of the principal foliation $T_i$ is a curvature surface on $U$.  Curvature surfaces are
plentiful, since the results of Reckziegel \cite{Reck2} and Singley \cite{Sin} imply that there is an open dense subset 
$\Omega$ of $M^{n-1}$ on which the multiplicities of the curvature spheres are locally constant.  On $\Omega$, each
leaf of each principal foliation is a curvature surface.

It is also possible to have a curvature surface $S$ which is not a leaf of a principal foliation, because the
multiplicity of the corresponding curvature sphere is not constant on a neighborhood of $S$, as in the following example
of Pinkall \cite{P4}.

\begin{example}
\label{ex:3.4.6}
A curvature surface that is not a leaf of a principal foliation.\\
\noindent
{\em Let $T^2$ be a torus of revolution 
in ${\bf R}^3$, and embed ${\bf R}^3$ into ${\bf R}^4 = {\bf R}^3 \times {\bf R}$.
Let $\eta$ be a field of unit normals to $T^2$ in ${\bf R}^3$.  Let $M^3$ be a 
tube of sufficiently small radius
$\varepsilon > 0$ around $T^2$ in ${\bf R}^4$, so that $M^3$ is a compact smooth embedded hypersurface in ${\bf R}^4$.
The normal space 
to $T^2$ in ${\bf R}^4$ at a point $x \in T^2$ is spanned by $\eta(x)$ and $e_4 = (0,0,0,1)$.
The shape operator $A_\eta$ of $T^2$ has two distinct principal curvatures at each point of $T^2$, while
the shape operator $A_{e_4}$ of $T^2$ is identically zero.  Thus the shape operator $A_\zeta$ for the normal
\begin{displaymath}
\zeta = \cos \theta \ \eta (x) + \sin \theta \ e_4,
\end{displaymath}
at a point $x \in T^2$, is given by 
\begin{displaymath}
A_\zeta = \cos \theta \ A_{\eta(x)}.
\end{displaymath}

\noindent
From the formulas for the 
principal curvatures of a tube 
(see Cecil--Ryan \cite[pp. 17--18]{CR8}), one finds that at all points of $M^3$ where
$x_4 \neq \pm \varepsilon$, there are three distinct principal curvatures of multiplicity one, which are constant 
along their corresponding lines of curvature 
(curvature surfaces of dimension one). One of these principal curvatures is $\mu = -1/\varepsilon$ resulting
from the tube construction.
However, on the two tori,
$T^2 \times \{ \pm \varepsilon\}$, the principal curvature $\kappa = 0$ has multiplicity two.  These two tori
are curvature surfaces for this principal curvature $\kappa$, since the principal space corresponding to $\kappa$ is
tangent to each torus at every point.  These two tori are not leaves of a principal foliation, however, since
the leaves of a foliation must all have the same dimension.
The Legendre lift $\lambda$ of 
this embedding of $M^3$ in ${\bf R}^4$ has the same properties.}
\end{example}

Part (e) of Corollary \ref{cor:3.4.5} 
has the following generalization, the proof of which is obtained by invoking the theorem of 
Ryan \cite{Ryan1} mentioned in the proof of Corollary \ref{cor:3.4.5}, with obvious minor modifications.

\begin{corollary}
\label{cor:3.4.7} 
Suppose that $S$ is a curvature surface of dimension $m > 1$ in a Legendre submanifold.  Then the corresponding curvature
sphere is constant along $S$.
\end{corollary}

As we stated at the beginning of the paper,
an oriented hypersurface $f:M^{n-1} \rightarrow S^n$ (or ${\bf R}^n$) is said to be {\em Dupin} if:

\begin{enumerate}
\item[(a)] along each curvature surface, the corresponding principal
curvature is constant.
\end{enumerate}
The hypersurface $M$ is called {\em proper Dupin} if, in addition 
to Condition (a), the following condition is satisfied:

\begin{enumerate}
\item[(b)] the number $g$ of distinct principal curvatures is constant on
$M$.
\end{enumerate}

On an open subset $U$ on which Condition (b) holds, Condition (a) is equivalent
to requiring that each curvature surface in each principal
foliation be an open subset of a metric sphere in $S^n$
of dimension equal to the multiplicity of the corresponding
principal curvature. Condition (a) is also equivalent to the condition that
along each curvature surface, the corresponding curvature sphere map is constant. Finally, on $U$, Condition (a)
is equivalent to requiring that for each principal curvature $\kappa$, the image of the focal
map $f_{\kappa}$ is a smooth submanifold of $S^n$ of codimension $m+1$, where $m$ is the multiplicity of $\kappa$.
See Cecil--Ryan \cite[pp. 18--34]{CR8} for proofs of these results.

One consequence of the results given above
is that like isoparametric hypersurfaces, all proper Dupin hypersurfaces are algebraic.
For simplicity, we take the ambient manifold to be ${\bf R}^n$.
The theorem was formulated by Cecil, Chi and Jensen \cite{CCJ3} as follows.
\begin{theorem}
\label{algebraic}
Every connected proper Dupin hypersurface 
$f:M \rightarrow{\bf R}^n$ embedded in ${\bf R}^n$ is
contained in a connected component of an 
irreducible algebraic subset of ${\bf R}^n$ of dimension $n-1$.  
\end{theorem}

Pinkall \cite{P6} sent the author a letter in 1984 that contained a sketch of a proof of this result.
However, a proof was not published until 2008 by Cecil, Chi and Jensen \cite{CCJ3}, who used methods of real
algebraic geometry to give a complete proof based on Pinkall's sketch. The proof makes use of the various principal foliations whose leaves are open subsets of spheres to construct an analytic algebraic parametrization of a 
neighborhood of $f(x)$ for each point $x \in M$. 

In contrast to the situation for isoparametric hypersurfaces, however, a connected proper Dupin hypersurface 
in $S^n$ does not necessarily lie in a compact connected proper Dupin hypersurface,
because the multiplicities of the principal curvatures are not necessarily constant on a compact Dupin hypersurface
that contains the proper Dupin hypersurface (see
Remark \ref{rem:3.4.6} below for an example).

An important class of proper Dupin hypersurfaces consists of the isoparametric hypersurfaces in 
$S^n$, and those hypersurfaces in ${\bf R}^n$ obtained from isoparametric hypersurfaces in $S^n$ via
stereographic projection.  For example, the well-known ring cyclides of Dupin \cite{D} in ${\bf R}^3$ are obtained
this way from a standard product torus $S^1(r) \times S^1(s) \subset S^3$, $r^2+s^2=1$.
 
The torus of revolution in ${\bf R}^3$ in Example \ref{ex:3.4.6} is a cyclide of Dupin.
On the torus,
there are $g=2$ distinct principal curvatures at each point, and each principal curvature is constant
along each leaf of its corresponding principal foliation.  These leaves are latitude circles for one
principal curvature and longitude circles for the other principal curvature.

\begin{remark}
\label{rem:3.4.6}
A Dupin hypersurface that is not proper Dupin.\\
\noindent
{\em The tube $M^3 \subset {\bf R}^4$
over the torus in Example \ref{ex:3.4.6} is an example of a Dupin hypersurface that is not proper Dupin.
At points of $M^3$ except those on the top and bottom tori $T^2 \times \{ \pm \varepsilon\}$, there are three distinct
principal curvatures that are each constant along their corresponding principal curves (which are circles).
However, on $T^2 \times \{ \pm \varepsilon\}$, there are only two distinct principal curvatures, $\kappa = 0$
of multiplicity two, and $\mu = -1/\varepsilon$ of multiplicity one.  Thus, $M^3$ is not proper Dupin, since
the number of distinct principal curvatures is not constant on $M^3$.  The hypersurface $M^3$ is Dupin, however,
since along each curvature surface (including $T^2 \times \{ \pm \varepsilon\}$), the corresponding
principal curvature is constant.}
\end{remark}

We generalize these definitions to the context of Lie sphere geometry
by defining a Legendre submanifold $\lambda: M^{n-1} \rightarrow \Lambda^{2n-1}$ to be a {\em Dupin submanifold} if:

\begin{enumerate}
\item[(a)] along each curvature surface, the corresponding 
curvature sphere is constant.
\end{enumerate}
The Legendre submanifold $\lambda$ is called {\em proper Dupin} if, in addition 
to Condition (a), the following condition is satisfied:

\begin{enumerate}
\item[(b)] the number $g$ of distinct curvature spheres is constant on
$M$.
\end{enumerate}

Of course, the Legendre lift of a
Dupin hypersurface in $S^n$ or ${\bf R}^n$ is
Dupin in the sense defined here, but the definition here is more general, because the spherical projection 
of a Dupin
submanifold need not be an immersion.  Corollary \ref{cor:3.4.7} shows that the only curvature surfaces which must
be considered in checking the Dupin property (a) are those of dimension one.  

The Legendre lift of the 
torus of revolution $T^2 \subset {\bf R}^3$ in Example \ref{ex:3.4.6}
above is a proper Dupin submanifold.  On the other hand, the Legendre lift of the 
tube $M^3$ over $T^2$ is
Dupin, but not proper Dupin, since the number of distinct curvature spheres is not constant on $M^3$. 

The following theorem shows that both the 
Dupin and proper Dupin conditions are invariant under Lie sphere transformations,
and many important classification results for Dupin submanifolds have been obtained in the
setting of Lie sphere geometry.

\begin{theorem}
\label{thm:dupin-lie-invariant} 
Let $\lambda: M^{n-1} \rightarrow \Lambda^{2n-1}$ be a Legendre submanifold and $\beta$ a Lie sphere transformation.\\
(a) If $\lambda$ is Dupin, then $\beta \lambda$ is Dupin.\\  
(b) If $\lambda$ is proper Dupin, then $\beta \lambda$ is proper Dupin.
\end{theorem}

\begin{proof}
By Theorem~\ref{thm:3.4.3}, a point $[K]$ on the line $\lambda (x)$ is a curvature sphere of $\lambda$ at $x \in M$
if and only if the point $\beta [K]$ is a curvature sphere of $\beta \lambda$ at $x$, and the principal
spaces corresponding $[K]$ and  $\beta [K]$ are identical.  Since these principal spaces are the same, if $S$ is
a curvature surface of $\lambda$ corresponding to a curvature sphere map $[K]$, then $S$ is also
a curvature surface of $\beta \lambda$ corresponding to a curvature sphere map $\beta [K]$, and clearly
$[K]$ is constant along $S$ if and only if $\beta [K]$ is constant along $S$.  This proves part (a) of the theorem.
Part (b) also follows immediately from Theorem~\ref{thm:3.4.3}, since for each $x \in M$, the number $g$ of
distinct curvature spheres of $\lambda$ at $x$ equals the number of distinct curvatures spheres of $\beta \lambda$
at $x$.  So if this number $g$ is constant on $M$ for $\lambda$, then it is constant on $M$ for $\beta \lambda$.
\end{proof}

\subsection{Local constructions of Dupin hypersurfaces}
\label{sec:4.1}
Pinkall \cite{P4} introduced four constructions
for obtaining a Dupin hypersurface $W$ in ${\bf R}^{n+m}$
from a Dupin hypersurface $M$ in ${\bf R}^n$.  We first describe these constructions in the case $m=1$ as follows.

Begin with a Dupin hypersurface
$M^{n-1}$ in ${\bf R}^n$ and then consider ${\bf R}^n$
as the linear subspace ${\bf R}^n \times \{ 0 \}$
in ${\bf R}^{n+1}$.  The following 
constructions yield a Dupin hypersurface $W^n$ in
${\bf R}^{n+1}$.
\begin{enumerate}
\item[(1)] Let $W^n$ be the cylinder $M^{n-1} \times {\bf R}$ in
${\bf R}^{n+1}$.
\item[(2)] Let $W^n$ be the hypersurface in ${\bf R}^{n+1}$
obtained by rotating
$M^{n-1}$ around an axis
${\bf R}^{n-1} \subset {\bf R}^n$.
\item[(3)] Let $W^n$ be a tube of constant radius in ${\bf R}^{n+1}$ around $M^{n-1}$.
\item[(4)] Project $M^{n-1}$ stereographically onto a hypersurface
$V^{n-1} \subset S^n \subset {\bf R}^{n+1}$.  Let
$W^n$ be the cone over $V^{n-1}$ in ${\bf R}^{n+1}$.
\end{enumerate}

In general, these constructions introduce a new principal curvature 
of multiplicity one which is constant along its lines 
of curvature.  The other principal curvatures are determined by the 
principal curvatures of $M^{n-1}$, and the Dupin property is preserved
for these principal curvatures.  These constructions can be
generalized to produce a new principal curvature of multiplicity
$m$ by considering ${\bf R}^n$ as a subset of ${\bf R}^n \times
{\bf R}^m$ rather than ${\bf R}^n \times {\bf R}$. (See \cite[pp. 127--141]{Cec1} for a thorough treatment of
these constructions in the context of Lie sphere geometry.)\\

Although Pinkall defined these four constructions, his Theorem 4 \cite[p. 438]{P4} 
showed that the cone construction is redundant in the context of Lie sphere geometry, since it
is Lie equivalent 
to a tube (see Remark \ref{rem:4.2.10} below). 
For this reason, we will restrict our attention to the
three standard constructions: tubes, cylinders and surfaces of revolution.

A Dupin submanifold obtained from a lower-dimensional Dupin submanifold via one of these standard constructions is said
to be {\em reducible}.  
More generally, a Dupin submanifold which is locally Lie equivalent to such a Dupin submanifold 
is called {\em reducible}.
Pinkall \cite[p. 438]{P4} proved the following Lie geometric characterization of reducibility which has been very useful
in the study of Dupin hypersurfaces.

\begin{theorem}
\label{thm:4.2.9} 
A connected proper Dupin submanifold $\lambda: W^{n-1} \rightarrow \Lambda^{2n-1}$ is reducible if and only if there exists a curvature sphere map $[K]$ of $\lambda$ that lies in a linear subspace of ${\bf P}^{n+2}$ of codimension two.
\end{theorem}

\begin{proof}
We first note that the following manifolds of spheres are hyperplane sections of the Lie quadric $Q^{n+1}$:\\

a) the hyperplanes in ${\bf R}^n$,

b) the spheres with a fixed signed radius $r$ in ${\bf R}^n$,

c) the spheres that are orthogonal to a fixed sphere in ${\bf R}^n$.\\

To see this, we use the Lie coordinates given in equation (\ref{eq:1.3.4}).
In Case a), the hyperplanes are characterized by the equation $x_1 + x_2 = 0$, which clearly determines a hyperplane section of $Q^{n+1}$.  In Case b), the spheres with signed radius $r$ are determined by the linear equation
\begin{displaymath}
r(x_1 + x_2) = x_{n+3}.
\end{displaymath}
In Case c), it can be assumed that the fixed sphere is a hyperplane $H$ through the origin in ${\bf R}^n$.  A sphere is orthogonal to $H$ if and only if its center lies in $H$.  This clearly imposes a linear condition on the vector in 
equation (\ref{eq:1.3.4}) representing the sphere.

The sets a), b), c) are each of the form
\begin{displaymath}
\{x \in Q^{n+1} \mid \langle x, w \rangle = 0\},
\end{displaymath}
with $\langle w,w \rangle = 0, -1, 1$ in Cases a), b), c), respectively.

We can now see that every reducible Dupin hypersurface has a family of curvature spheres that is contained in two hyperplane sections of the Lie quadric as follows.

For the cylinder construction, the tangent hyperplanes of the cylinder are curvature spheres that are orthogonal to a fixed hyperplane in ${\bf R}^n$.  Thus, that family of curvature spheres is contained in a $n$-dimensional
linear subspace $E$ of  ${\bf P}^{n+2}$ such that the signature of $\langle \ ,\ \rangle$ on the polar subspace $E^\perp$ of $E$ is $(0,+)$.

For the surface of revolution construction, the spheres in the
new family of curvature spheres all have their centers in the axis of revolution, which is a linear subspace of codimension 2 in ${\bf R}^n$.  Thus, that family of curvature spheres is contained in a $n$-dimensional linear subspace $E$ of  ${\bf P}^{n+2}$ such that the signature of $\langle \ ,\ \rangle$ on the polar subspace $E^\perp$ of $E$ is $(+,+)$.

For the tube construction, the spheres in the
the new family of curvature spheres all have the same radius, and their centers all lie in the hyperplane of ${\bf R}^n$ containing the manifold over which the tube is constructed.
Thus, that family of curvature spheres is contained in a $n$-dimensional linear subspace $E$ of  ${\bf P}^{n+2}$ such that the signature of $\langle \ ,\ \rangle$ on the polar subspace $E^\perp$ of $E$ is $(-,+)$.

Conversely, suppose that $K: W^{n-1} \rightarrow {\bf P}^{n+2}$ is a family of curvature spheres that is contained
in an $n$-dimensional linear subspace $E$ of  ${\bf P}^{n+2}$.  Then  $\langle \ ,\ \rangle$  must have signature
$(+,+)$, $(0,+)$ or $(-,+)$ on the polar subspace $E^\perp$, because otherwise $E \cap Q^{n+1}$ would be empty or would consist of a single point.

If the signature of $E^\perp$ is $(+,+)$, then there exists a Lie sphere transformation $A$ which takes  $E$ 
to a space  $F = A(E)$ such that $F \cap Q^{n+1}$ consists of all spheres that have their centers in a fixed $(n-2)$-dimensional linear subspace ${\bf R}^{n-2}$ of 
${\bf R}^n$.  Since one family of curvature spheres of this Dupin submanifold $A\lambda$ lies in $F \cap Q^{n+1}$, and the Dupin submanifold $A\lambda$ is the envelope of these spheres,  $A\lambda$ must be a surface of revolution with the axis ${\bf R}^{n-2}$ (see \cite[pp. 142--143]{Cec1} for more detail on envelopes of
families of spheres in this situation), and so $\lambda$ is reducible.

If the signature of $E^\perp$ is $(0,+)$, then there exists a Lie sphere transformation $A$ which takes  $E$ 
to a space  $F = A(E)$ such that $F \cap Q^{n+1}$ consists of hyperplanes orthogonal to a fixed hyperplane in 
${\bf R}^n$.  Since one family of curvature spheres of this Dupin submanifold $A\lambda$ lies in $F \cap Q^{n+1}$, and the Dupin submanifold $A\lambda$ is the envelope of these spheres,  $A\lambda$ is obtained as a result of the cylinder construction, and so $\lambda$ is reducible.

If the signature of $E^\perp$ is $(-,+)$, then there exists a Lie sphere transformation $A$ which takes  $E$ 
to a space  $F = A(E)$ such that $F \cap Q^{n+1}$ 
consists of spheres that all have the same radius and whose centers lie in a hyperplane
${\bf R}^{n-1}$ of ${\bf R}^n$.  Since one family of curvature spheres of this Dupin submanifold $A\lambda$ lies in $F \cap Q^{n+1}$, and the Dupin submanifold $A\lambda$ is the envelope of these spheres,  $A\lambda$ is obtained as a result of the tube construction, and so $\lambda$ is reducible.
\end{proof}

\begin{remark}
\label{rem:4.2.10} 
{\em When Pinkall introduced his constructions, he also listed the following cone construction.  
Begin with a proper Dupin
submanifold $\lambda$ induced by an embedded proper Dupin hypersurface $M^{n-1} \subset S^n \subset {\bf R}^{n+1}$.
The new Dupin submanifold $\mu$ is the Legendre submanifold induced from the cone 
$C^n$ over
$M^{n-1}$ in ${\bf R}^{n+1}$ with vertex at the origin.  Theorem \ref{thm:4.2.9} shows that this construction is
locally Lie equivalent to the tube construction as follows.  As shown in Theorem \ref{thm:4.2.9}, the
tube construction is characterized by the fact
that one curvature sphere map $[K]$ lies in a $n$-dimensional linear subspace $E$ of ${\bf P}^{n+2}$, whose 
orthogonal complement has signature $(-,+)$.  
For the cone construction, the new family $[K]$ of curvature spheres consists
of hyperplanes 
through the origin (the point $[e_1+e_2]$ in Lie sphere geometry) 
that are tangent to the cone along the rulings.  Since the hyperplanes all
pass through the improper point 
$[e_1-e_2]$ as well as the origin $[e_1+e_2]$, they correspond to points in the linear subspace $E$, whose orthogonal
complement is as follows:
\begin{displaymath}
E^{\perp} = {\mbox {\rm Span }}\{e_1+e_2, e_1-e_2\}.
\end{displaymath}
Since $E^{\perp}$ is spanned by $e_1$ and $e_2$, it has signature $(-,+)$.  Thus, the cone construction is
Lie equivalent to the tube construction.} 
\end{remark}

Using these constructions, Pinkall was able to produce a proper Dupin hypersurface 
in Euclidean space with an
arbitrary number of distinct principal curvatures, each with any given multiplicity (see Theorem \ref{thm:4.1.1}
below).  In general, the proper Dupin hypersurfaces constructed using Pinkall's constructions
cannot be extended to compact Dupin hypersurfaces without losing the
property that the number of distinct principal curvatures is constant (see  \cite[pp. 127--141]{Cec1} for more detail).
For now, we give a
proof of Pinkall's theorem without attempting to compactify the hypersurfaces constructed.

\begin{theorem}
\label{thm:4.1.1} 
Given positive integers $m_1,\ldots,m_g$ with 
\begin{displaymath}
m_1 + \cdots + m_g = n-1,
\end{displaymath}
there exists a proper Dupin hypersurface
in ${\bf R}^n$ with $g$ distinct principal curvatures having respective multiplicities $m_1,\ldots,m_g$.
\end{theorem}
\begin{proof}
The proof is by an inductive construction, which will be clear once the first few examples are done.  To begin, note
that a usual torus of revolution $T^2$
in ${\bf R}^3$ is a proper Dupin hypersurface with two principal curvatures of multiplicity one.  To
construct a proper Dupin hypersurface $W^3$ in ${\bf R}^4$ with three principal curvatures, each of multiplicity one,
begin with an open subset $U$ of a torus of revolution in ${\bf R}^3$ on which neither principal curvature vanishes.
Take $W^3$ to be the cylinder 
$U \times {\bf R}$ in ${\bf R}^3 \times {\bf R} = {\bf R}^4$.  Then $W^3$ has three
distinct principal curvatures at each point, one of which is zero.  These are clearly constant along their
corresponding 1-dimensional curvature surfaces.

To get a proper Dupin hypersurface in ${\bf R}^5$ with three principal curvatures having respective multiplicities
$m_1 = m_2 = 1$, $m_3 =2$, one simply takes 
\begin{displaymath}
U \times {\bf R}^2 \subset {\bf R}^3 \times {\bf R}^2 = {\bf R}^5,
\end{displaymath}
where $U$ is set defined above.
To obtain a proper Dupin hypersurface $Z^4$ in ${\bf R}^5$ with four principal curvatures of multiplicity one, 
first invert the hypersurface
$W^3$ above in a 3-sphere in ${\bf R}^4$, chosen so that the image of $W^3$ contains an open subset $V^3$ on which no
principal curvature vanishes.  The hypersurface $V^3$ is proper Dupin, since the proper Dupin property is preserved by
M\"{o}bius transformations.  Now take $Z^4$ to be the cylinder
$V^3 \times {\bf R}$ in ${\bf R}^4 \times {\bf R} = {\bf R}^5$.
\end{proof}

The proof of this theorem gives an indication of the type of problems that occur when attempting
to extend these constructions to produce a compact, proper Dupin hypersurface.  In particular,
for the cylinder construction, the new principal curvature on the constructed hypersurface
$W^n$ is identically zero.  Thus, in order for $W^n$ to be proper Dupin, 
either zero is not a principal curvature at any point 
of the original hypersurface $M^{n-1}$, or else zero is a principal curvature of
constant multiplicity on $M^{n-1}$.  Otherwise, the principal curvature zero will not have constant
multiplicity on $W^n$, which implies that $W^n$ is not proper Dupin.

\subsection{Lie curvatures of Dupin hypersurfaces}
\label{sec:3.5}
In this section, we introduce certain natural Lie invariants
of Legendre submanifolds which have been useful
in the study of Dupin and isoparametric hypersurfaces.

Let $\lambda: M^{n-1} \rightarrow \Lambda^{2n-1}$ be an arbitrary Legendre submanifold.  As before, we can write
$\lambda = [Y_1, Y_2]$ with
\begin{equation}
\label{eq:3.5.1}
Y_1 = (1, f, 0), \quad Y_2 = (0, \xi, 1),
\end{equation} 
where $f$ and $\xi$ are the spherical projection and spherical field of unit normals, respectively.
At each point
$x \in M^{n-1}$, the points on the line $\lambda(x)$ can be written in the form,
\begin{equation}
\label{eq:3.5.2}
\mu Y_1 (x) +  Y_2 (x),
\end{equation} 
i.e., take $\mu$ as an inhomogeneous coordinate along the projective line $\lambda(x)$. Note that $k_1$
corresponds to $\mu = \infty$.  

The next two theorems give the relationship between the coordinates of the 
curvature spheres of $\lambda$ and the 
principal curvatures of $f$, in the case where 
$f$ has constant rank.  In the first
theorem, we assume that the spherical projection $f$ is an immersion on $M^{n-1}$.  By Theorem~\ref{thm:3.4.4},
we know that this can always be achieved locally by passing to a parallel submanifold.

\begin{theorem}
\label{thm:3.5.1} 
Let $\lambda: M^{n-1} \rightarrow \Lambda^{2n-1}$ be a Legendre submanifold whose spherical projection 
$f:M^{n-1} \rightarrow S^n$ is an immersion.  Let $Y_1$ and $Y_2$ be the point sphere and great sphere maps
of $\lambda$ as in equation (\ref{eq:3.5.1}).  Then the curvature spheres of $\lambda$ at a point $x \in M^{n-1}$ are
\begin{displaymath}
[K_i] = [\kappa_i Y_1 + Y_2], \quad 1 \leq i \leq g,
\end{displaymath}
where $\kappa_1,\ldots,\kappa_g$ are the distinct principal curvatures at $x$ of the oriented hypersurface $f$
with field of unit normals $\xi$.  The multiplicity of the curvature sphere $[K_i]$ equals the multiplicity
of the principal curvature $\kappa_i$.
\end{theorem}
\begin{proof}
Let $X$ be a nonzero vector in $T_xM^{n-1}$.  Then for any real number $\mu$,
\begin{displaymath}
d(\mu Y_1  +  Y_2)(X) = (0, \mu \ df(X) + d\xi(X), 0). 
\end{displaymath}
This vector is in Span $\{Y_1(x), Y_2(x)\}$ if and only if
\begin{displaymath}
\mu \ df(X) + d\xi(X) = 0, 
\end{displaymath}
i.e., $\mu$ is a principal curvature of $f$ with corresponding principal vector $X$.
\end{proof}

A second noteworthy case is when the point sphere map $Y_1$ is a curvature sphere of constant multiplicity $m$ on
$M^{n-1}$.  By Corollary \ref{cor:3.4.5}, the corresponding principal distribution is a foliation, and the 
curvature sphere map $[Y_1]$ is constant along the leaves of this foliation.  Thus the map $[Y_1]$ factors through
an immersion $[W_1]$ from the space of leaves 
$V$ of this foliation into $Q^{n+1}$.  We can write
\begin{displaymath}
W_1 = (1, \phi, 0),
\end{displaymath}
where $\phi:V \rightarrow S^n$ is an immersed submanifold of codimension $m+1$. The manifold $M^{n-1}$ is locally
diffeomorphic to an open subset of the unit normal 
bundle $B^{n-1}$ of the submanifold $\phi$, and $\lambda$ is essentially
the Legendre lift of the submanifold $\phi (V)$, as defined on page \pageref{Legendre-lift}.  
The following theorem relates the curvature spheres
of $\lambda$ to the principal curvatures of $\phi$.  Recall that the point sphere and great sphere maps for $\lambda$
are given as in equation (\ref{eq:3.3.3}) by
\begin{equation}
\label{eq:3.5.3}
Y_1(x,\xi) = (1, \phi(x), 0), \quad Y_2(x,\xi) = (0, \xi, 1),
\end{equation}
for $(x,\xi) \in B^{n-1}$.
The proof of Theorem \ref{thm:3.5.2}  is similar to the proof of Theorem \ref{thm:3.5.1}  above, but it requires putting local coordinates on the unit normal bundle $B^{n-1}$, and we omit it here (see \cite[p. 74]{Cec1}).

\begin{theorem}
\label{thm:3.5.2} 
Let $\lambda: B^{n-1} \rightarrow \Lambda^{2n-1}$ be the Legendre lift
of the immersed submanifold $\phi(V)$ in $S^n$
of codimension $m+1$.  Let $Y_1$ and $Y_2$ be the point sphere and great sphere maps
of $\lambda$ as in equation (\ref{eq:3.5.3}).  Then the curvature spheres of $\lambda$ at a point $(x, \xi) \in B^{n-1}$ are
\begin{displaymath}
[K_i] = [\kappa_i Y_1 + Y_2], \quad 1 \leq i \leq g,
\end{displaymath}
where $\kappa_1,\ldots,\kappa_{g-1}$ are the distinct principal curvatures of the shape operator $A_\xi$,
and $\kappa_g = \infty$.  For $1 \leq i \leq g-1$, the multiplicity of the curvature sphere $[K_i]$ equals the multiplicity
of the principal curvature $\kappa_i$, while the multiplicity of $[K_g]$ is $m$.
\end{theorem}

Given these two theorems, we define a 
{\em principal curvature} of a Legendre submanifold
$\lambda: M^{n-1} \rightarrow \Lambda^{2n-1}$ at a point $x \in M^{n-1}$ to be a value $\kappa$ in the set
${\bf R} \cup \{\infty \}$ such that $[\kappa Y_1(x) + Y_2 (x)]$ is a curvature sphere of $\lambda$ at $x$,
where $Y_1$ and $Y_2$ are as in equation (\ref{eq:3.5.1}).

These principal curvatures of a Legendre submanifold are not Lie invariant, and 
they depend on the special parametrization
for $\lambda$ given in equation (\ref{eq:3.5.1}).  However, Miyaoka 
\cite{Mi3} pointed out that the cross-ratios of the
principal curvatures are Lie invariant.  

In order to formulate Miyaoka's theorem, we need to introduce some
notation.  Suppose that $\beta$ is a Lie sphere transformation.  The Legendre submanifold $\beta \lambda$ has
point sphere and great sphere maps given, respectively, by
\begin{displaymath}
Z_1 = (1, h, 0), \quad Z_2 = (0, \zeta, 1),
\end{displaymath} 
where $h$ and $\zeta$ are the spherical projection
and spherical field of unit normals for $\beta \lambda$.  Suppose that
\begin{displaymath}
[K_i] = [\kappa_i Y_1 + Y_2], \quad 1 \leq i \leq g,
\end{displaymath}
are the distinct curvature spheres of $\lambda$ at a point $x \in M^{n-1}$.  By Theorem~\ref{thm:3.4.3}, the points
$\beta [K_i], 1 \leq i \leq g$, are the distinct curvature spheres of $\beta \lambda$ at $x$.  We can write
\begin{displaymath}
\beta [K_i] = [\gamma_i Z_1 + Z_2], \quad 1 \leq i \leq g.
\end{displaymath}
These $\gamma_i$ are the principal curvatures of $\beta \lambda$ at $x$.

For four distinct numbers $a,b,c,d$ in ${\bf R} \cup \{\infty\}$, we adopt the notation
\begin{equation}
\label{eq:3.5.4}
[a,b;c,d] = \frac{(a-b)(d-c)}{(a-c)(d-b)}
\end{equation} 
for the cross-ratio of $a,b,c,d$.  We use the usual conventions involving operations with $\infty$.
For example, if $d= \infty$, then the expression $(d-c)/(d-b)$ evaluates to one,
and the cross-ratio $[a,b;c,d]$ equals $(a-b)/(a-c)$.

Miyaoka's theorem can now be stated as follows.

\begin{theorem}
\label{thm:3.5.4} 
Let $\lambda: M^{n-1} \rightarrow \Lambda^{2n-1}$ be a Legendre submanifold and $\beta$ a Lie sphere transformation.
Suppose that $\kappa_1,\ldots,\kappa_g, g \geq 4,$ are the distinct principal curvatures of $\lambda$ at a point
$x \in M^{n-1}$, and $\gamma_1,\ldots,\gamma_g$ are the corresponding principal curvatures of $\beta \lambda$ at
$x$.  Then for any choice of four numbers $h,i,j,k$ from the set $\{1,\ldots,g\}$, we have
\begin{equation}
\label{eq:3.5.5}
[\kappa_h,\kappa_i;\kappa_j, \kappa_k] = [\gamma_h,\gamma_i;\gamma_j, \gamma_k].
\end{equation} 
\end{theorem}
\begin{proof}
The left side of equation (\ref{eq:3.5.5}) is the cross-ratio, in the sense of projective geometry, of the four points
$[K_h],[K_i],[K_j],[K_k]$ on the projective line $\lambda(x)$.  The right side of equation (\ref{eq:3.5.5}) 
is the cross-ratio
of the images of these four points under $\beta$.  The theorem now follows from the fact that the projective
transformation $\beta$ preserves the cross-ratio of four points on a line.
\end{proof}

The cross-ratios of the principal curvatures of $\lambda$ are called the 
{\em Lie curvatures} of $\lambda$.  A set of
related invariants for the M\"{o}bius group 
is obtained as follows.  First, recall that a M\"{o}bius transformation
is a Lie sphere transformation that takes point spheres to point spheres.  Hence the transformation $\beta$
in Theorem~\ref{thm:3.5.4} is a M\"{o}bius transformation 
if and only if $\beta [Y_1] = [Z_1]$.  This leads to the following corollary of Theorem~\ref{thm:3.5.4}.

\begin{corollary}
\label{cor:3.5.5} 
Let $\lambda: M^{n-1} \rightarrow \Lambda^{2n-1}$ be a Legendre submanifold and $\beta$ a M\"{o}bius transformation.
Then for any three distinct principal curvatures $\kappa_h,\kappa_i,\kappa_j$ of $\lambda$ at a point
$x \in M^{n-1}$, none of which equals $\infty$, we have
\begin{equation}
\label{eq:3.5.6}
\Phi(\kappa_h,\kappa_i,\kappa_j) = (\kappa_h - \kappa_i)/(\kappa_h - \kappa_j) = (\gamma_h - \gamma_i)/(\gamma_h - \gamma_j),
\end{equation}
where $\gamma_h, \gamma_i$ and $\gamma_j$ are the corresponding principal curvatures of  $\beta \lambda$ at the point $x$.
\end{corollary}
\begin{proof}
First, note that we are using equation (\ref{eq:3.5.6}) to define the quantity $\Phi$.  Now since $\beta$ is a
M\"{o}bius transformation, the point $[Y_1]$, corresponding to $\mu = \infty$, is taken by $\beta$ to the point
$Z_1$ with coordinate $\gamma = \infty$.  Since $\beta$ preserves cross-ratios, we have

\begin{equation}
\label{eq:moeb-curv}
[\kappa_h,\kappa_i;\kappa_j, \infty] = [\gamma_h,\gamma_i;\gamma_j, \infty].
\end{equation} 
The corollary now follows since the cross-ratio on the left in (\ref{eq:moeb-curv}) equals the left side of equation
(\ref{eq:3.5.6}), and the cross-ratio on the right in (\ref{eq:moeb-curv}) equals the right side of equation 
(\ref{eq:3.5.6}).
\end{proof}

A ratio $\Phi$ of the form (\ref{eq:3.5.6}) is called a 
{\em M\"{o}bius curvature} of $\lambda$.  Lie and M\"{o}bius
curvatures have been useful in characterizing Legendre submanifolds that are 
Lie equivalent to Legendre lifts of
isoparametric hypersurfaces in spheres, as we will see in the next section.

\section{Classifications of Dupin Hypersurfaces}
\label{chap:2}
Many of the important classification results for Dupin hypersurfaces involve conditions under which the
Dupin hypersurface is Lie equivalent to the Legendre lift of an isoparametric hypersurface in a sphere.  These results 
will be discussed in this section.  We begin with a brief summary of the classification results for
isoparametric hypersurfaces in spheres.

\subsection{Isoparametric hypersurfaces in spheres}
\label{sec:3.6}
An oriented hypersurface $M$ in a real space form ${\bf R}^n$, $S^n$ or $H^n$ is called an isoparametric hypersurface
if it has constant principal curvatures (see, for example, \cite[pp. 1--3]{CR8} for other characterizations).

An isoparametric hypersurface $M$ in ${\bf R}^n$ can have at most
two distinct principal curvatures, and $M$ must be an open 
subset of a hyperplane, hypersphere or a spherical cylinder
$S^k \times {\mathbf R}^{n-k-1}$.  
This was first proven for $n=3$ by Somigliana \cite{Som} in 1919
(see also Levi-Civita \cite{Lev} (1937) for $n=3$ and B. Segre \cite{Seg} (1938) for arbitrary $n$).

Shortly after the publication of the papers of Levi-Civita and Segre, 
Cartan \cite{Car2}--\cite{Car5} undertook the study of isoparametric hypersurfaces in arbitrary
real space-forms.
Cartan showed that an isoparametric hypersurface in ${\bf R}^n$ or $H^n$ can have at most two distinct principal curvatures, and he gave a local classification of isoparametric hypersurfaces in both ambient spaces
(see \cite[pp. 85--101]{CR8}).

In the sphere $S^n$, however, Cartan showed that there are many more possibilities.
He found examples of isoparametric hypersurfaces in $S^n$ with $1,2,3$ or 4 distinct principal curvatures, and he
classified connected isoparametric hypersurfaces with $g \leq 3$ principal curvatures as follows.  If $g=1$,
then the isoparametric hypersurface $M$ is totally umbilic, and it must be an open subset of a great or small
sphere.  If $g=2$, then $M$ must be an open subset of a standard product of two spheres,
\begin{displaymath}
S^k(r) \times S^{n-k-1}(s) \subset S^n, \quad r^2+s^2=1.
\end{displaymath}

In the case $g=3$, Cartan \cite{Car3}
showed that all the principal curvatures must have the same multiplicity
$m=1,2,4$ or 8, and the isoparametric hypersurface must be an open subset of a tube of
constant radius 
over a standard embedding of a projective
plane ${\bf F}P^2$ into $S^{3m+1}$ (see, for example, Cecil--Ryan \cite[pp. 151--155]{CR8}), 
where ${\bf F}$ is the division algebra
${\bf R}$, ${\bf C}$, ${\bf H}$ (quaternions),
${\bf O}$ (Cayley numbers), for $m=1,2,4,8,$ respectively.  Thus, up to
congruence, there is only one such family for each value of $m$.

Cartan's theory was further developed by Nomizu \cite{Nom3}--\cite{Nom4},
Takagi and Takahashi \cite{TT}, Ozeki and Takeuchi \cite{OT}, 
and most extensively by M\"{u}nzner \cite{Mu}--\cite{Mu2} (see also the English translations \cite{Mu}--\cite{Mu2}),
who showed that the number $g$ of distinct principal curvatures of an isoparametric hypersurface must be
$1,2,3,4$ or 6.
For more detail on the theory of isoparametric hypersurfaces, see
Chapter 3 of Cecil--Ryan 
\cite{CR8} or Chapter 3 of \cite{CR7}.  

In the case of an 
isoparametric hypersurface with four principal curvatures, 
M\"{u}nzner proved that
the principal curvatures can have at most two
distinct multiplicities $m_1,m_2$. Next Ferus, Karcher and M\"{u}nzner
\cite{FKM} (see also the English translation \cite{FKM})
used representations
of Clifford algebras to construct
for any positive integer
$m_1$ an infinite series of isoparametric hypersurfaces
with four principal curvatures having respective multiplicities
$(m_1,m_2)$, where $m_2$ is nondecreasing and
unbounded in each series.  (These examples are also described in
\cite[pp. 95--112]{Cec1} and \cite[pp. 162--180]{CR8}.)  

As later work by several researchers would show (see below), this class of
{\it FKM-type} isoparametric hypersurfaces contains all isoparametric
hypersurfaces with four principal curvatures with the exception of two
homogeneous examples, having multiplicities $(2,2)$ and
$(4,5)$.  
This construction of Ferus, Karcher and M\"{u}nzner was a generalization of an earlier construction due to Ozeki and Takeuchi \cite{OT}. 
  
Stolz \cite{Stolz} next proved that the multiplicities $(m_1,m_2)$
of the principal curvatures of an isoparametric
hypersurface with four principal curvatures
must be the same as those of the hypersurfaces of
FKM-type or the two homogeneous exceptions.
Cecil, Chi and Jensen \cite{CCJ1} then
showed that if the multiplicities of an isoparametric hypersurface with four principal curvatures satisfy $m_2 \geq 2 m_1 - 1$, then
the hypersurface is of FKM-type.  (A different proof of this result, using isoparametric 
triple systems, was given later by Immervoll \cite{Im}.)

Taken together with known results of Takagi \cite{Takagi} for $m_1 = 1$,
and Ozeki and Takeuchi \cite{OT} for $m_1 = 2$, this result of Cecil, Chi and Jensen handled all
possible pairs of multiplicities except for four cases, the homogeneous pair $(4,5)$, and the FKM pairs
$(3,4), (6,9)$ and $(7,8)$.
In a series of recent papers, Chi \cite{Chi}--\cite{Chi4} completed the classification of isoparametric hypersurfaces with four principal curvatures.
Specifically, Chi showed that in the cases $(3,4)$, $(6,9)$ and $(7,8)$, the isoparametric hypersurface must be of FKM-type, and in the case $(4,5)$, it must be homogeneous.

In the case of an isoparametric hypersurface with six principal curvatures, M\"{u}nzner showed
that all of the principal curvatures must have the same multiplicity $m$, and
Abresch \cite{Ab} showed that $m$ must equal 1 or 2.
By the classification of homogeneous isoparametric hypersurfaces due to Takagi and Takahashi \cite{TT}, 
there is only one homogeneous family in each case up to congruence.  These 
homogeneous examples have been shown to be
the only isoparametric hypersurfaces in the case $g=6$ by Dorfmeister and Neher \cite{DN5}
in the case of multiplicity $m=1$, and by Miyaoka \cite{Mi11}--\cite{Mi12} in the case $m=2$
(see also the papers of Siffert \cite{Siffert1}--\cite{Siffert2}).
See the surveys of Thorbergsson \cite{Th6}, Cecil \cite{Cec9}, and Chi \cite{Chi-survey} for more information on isoparametric hypersurfaces in spheres.

\subsection{Compact proper Dupin hypersurfaces}
\label{sec:3.7}

There is a close relationship between the theory of isoparametric hypersurfaces and the theory
of compact proper Dupin hypersurfaces embedded in $S^n$ (or ${\bf R}^n$).
Thorbergsson \cite{Th1} showed that the restriction $g = 1,2,3,4$ or 6 on the number of distinct principal
curvatures of an isoparametric hypersurface
also holds for a connected, compact proper Dupin hypersurface $M$ embedded in $S^n \subset {\bf R}^{n+1}$.
He first showed that $M$ must be {\em taut}, i.e., 
every nondegenerate distance function
$L_p (x) = d(p,x)^2$, where $d(p,x)$ is the distance from $p$ to $x$ in $S^n$,
has the minimum number of critical points required by the Morse inequalities 
on $M$.  Using tautness, he then showed that $M$ divides $S^n$ into two ball bundles over the first focal
submanifolds on either side of $M$.  This topological situation is what is required for 
M\"{u}nzner's proof of the restriction on $g$.  (See also \cite[65--74]{CR8}.)

M\"{u}nzner's argument also produces certain restrictions on the cohomology and the
homotopy groups of isoparametric hypersurfaces.  These restrictions necessarily apply to compact proper
Dupin hypersurfaces by Thorbergsson's result.  Grove and Halperin \cite{GH} 
later found more topological similarities
between these two classes of hypersurfaces.  Furthermore, the results of Stolz \cite{Stolz} 
and Grove-Halperin \cite{GH} on the possible multiplicities
of the principal curvatures actually only require the assumption that $M$ is a compact proper Dupin hypersurface, and not
that the hypersurface is isoparametric.

The close relationship between these two classes of hypersurfaces led to the widely held conjecture that every
compact proper Dupin hypersurface $M \subset S^n$ is Lie equivalent to an isoparametric
hypersurface (see \cite[p. 184]{CR7}).  We now describe the results that have been obtained concerning this 
conjecture. (See \cite{Cec10} for a full account.)

The conjecture is obviously true for $g=1$, in which case the compact proper Dupin hypersurface $M \subset S^n$
must be a hypersphere in $S^n$, and so $M$ itself is isoparametric.  In 1978, Cecil and Ryan \cite{CR2}
showed that if $g=2$, then $M$ must be an $(n-1)$-dimensional cyclide of Dupin, 
and it is therefore M\"{o}bius equivalent to an isoparametric
hypersurface in $S^n$.  Then in 1984, Miyaoka \cite{Mi1} showed that the conjecture is true for $g=3$, that is, $M$ must be
Lie equivalent an isoparametric hypersurface, although $M$ is not necessarily M\"{o}bius equivalent to an isoparametric hypersurface. Thus, as $g$ increases, the group needed 
to obtain equivalence with an isoparametric hypersurface gets progressively larger.  

For $g=4$, the conjecture remained open for several years until finally in 1988, 
counterexamples to the conjecture were discovered
independently by Pinkall and Thorbergsson \cite{PT1} 
and by Miyaoka and Ozawa \cite{MO}.  The two types of counterexamples are different, and the
method of Miyaoka and Ozawa also 
yields counterexamples to the conjecture with $g=6$ principal curvatures.  

For both the counterexmples of
Miyaoka-Ozawa and those of Pinkall-Thorbergsson,
the hypersurfaces constructed
do not have constant Lie curvatures, and so they cannot be Lie equivalent to an isoparametric
hypersurface. (See  also \cite[pp. 112--123]{Cec1} or \cite[pp. 308--322]{CR8} 
for a description of these counterexamples to the conjecture.)

This led to the following revised version of the conjecture due to Cecil, Chi and Jensen \cite{CCJ2}, \cite{CCJ4}
that includes the assumption of constant
Lie curvatures. 

\begin{conjecture}
\label{conj:CCJ-compact}
Every compact, connected proper Dupin hypersurface in $S^n$ with $g=4$ or $g=6$ principal curvatures and constant Lie
curvatures is Lie equivalent to an isoparametric hypersurface.
\end{conjecture}

Research on this revised conjecture has been an important part of the field, since the publication of the counterexamples 
to the original conjecture.  A useful result in this area is the following 
local Lie geometric characterization of 
Legendre submanifolds that are Lie equivalent
to the Legendre lift of
an isoparametric hypersurface in $S^n$ (see \cite{Cec4} or \cite[p. 77]{Cec1}).  

Recall that a line in ${\bf P}^{n+2}$ is called 
{\em timelike} if it contains only timelike points.
This means that an orthonormal basis for the 2-plane in ${\bf R}^{n+3}_2$ determined by the timelike line in 
${\bf P}^{n+2}$ consists
of two timelike vectors.  An example is the line $[e_1, e_{n+3}]$.

\begin{theorem}
\label{thm:3.5.6} 
Let $\lambda: M^{n-1} \rightarrow \Lambda^{2n-1}$ be a Legendre submanifold with $g$ distinct curvature spheres
$[K_1],\ldots,[K_g]$ at each point.  Then $\lambda$ is Lie equivalent 
to the Legendre lift of
an isoparametric
hypersurface in $S^n$ if and only if there exist $g$ points $[P_1],\ldots,[P_g]$ on a timelike line in ${\bf P}^{n+2}$
such that
\begin{displaymath}
\langle K_i,P_i \rangle = 0, \quad 1 \leq i \leq g.
\end{displaymath}
\end{theorem}
\begin{proof}
If $\lambda$ is the Legendre lift of an 
isoparametric hypersurface in $S^n$, then all the spheres in each family
$[K_i]$ have the same radius 
$\rho_i$, where $0 < \rho_i < \pi$.  By formula (\ref{eq:1.4.4}),
this is equivalent to the condition $\langle K_i,P_i \rangle = 0$, where the
\begin{equation}
\label{eq:3.5.7}
P_i = \sin \rho_i \ e_1 - \cos \rho_i \ e_{n+3}, \quad 1 \leq i \leq g,
\end{equation}
are $g$ points on the timelike line $[e_1, e_{n+3}]$.  Since a Lie sphere transformation preserves curvature spheres,
timelike lines and the polarity relationship, the same is true for any image of $\lambda$
under a Lie sphere transformation.

Conversely, suppose that there exist $g$ points $[P_1],\ldots,[P_g]$ on a timelike line $\ell$ such that
\begin{displaymath}
\langle K_i,P_i \rangle = 0, \quad 1 \leq i \leq g.
\end{displaymath}
Let $\beta$ be a Lie sphere transformation that maps $\ell$
to the line $[e_1, e_{n+3}]$.  Then the curvature spheres $\beta [K_i]$ of $\beta \lambda$ are respectively
orthogonal to the points $[Q_i] = \beta [P_i]$ on the line $[e_1, e_{n+3}]$.  This means that 
for each value of $i$, the spheres corresponding
to $\beta [K_i]$ have constant radius on $M^{n-1}$.  By applying a 
parallel transformation, if necessary, we can arrange 
that none of these curvature spheres has radius zero. 
Then $\beta \lambda$ is the Legendre lift of
an isoparametric hypersurface in $S^n$. 
\end{proof}

\begin{remark}
\label{rem:3.5.7}
{\em In the case where $\lambda$ is Lie equivalent
to the Legendre lift of 
an isoparametric hypersurface in $S^n$, one can say more about the position of the points $[P_1],\ldots,[P_g]$ on the timelike line $\ell$. 
M\"{u}nzner \cite{Mu}--\cite{Mu2} showed that the
radii $\rho_i$ of the curvature spheres of an isoparametric hypersurface must be of the form
\begin{equation}
\label{eq:3.5.8}
\rho_i = \rho_1 + (i-1) \frac {\pi}{g}, \quad 1 \leq i \leq g,
\end{equation}
for some $\rho_1 \in (0,\pi/g)$.  Hence, after Lie sphere transformation, the $[P_i]$ must have the form
(\ref{eq:3.5.7}) for $\rho_i$ as in equation (\ref{eq:3.5.8})}.
\end{remark}

Since the principal curvatures are constant on an isoparametric hypersurface, the Lie curvatures are also constant.
By M\"{u}nzner's work, the distinct principal curvatures $\kappa_i, 1 \leq i \leq g$, of an isoparametric hypersurface
must have the form
\begin{equation}
\label{eq:3.5.9}
\kappa_i = \cot \rho_i,
\end{equation}
for $\rho_i$ as in equation (\ref{eq:3.5.8}).  Thus the Lie curvatures of an isoparametric hypersurface can be determined.
We can order the principal curvatures so that 
\begin{equation}
\label{eq:3.5.9a}
\kappa_1 < \cdots < \kappa_g.
\end{equation}
In the case $g=4$, this leads to a unique Lie curvature $\Psi$ defined by
\begin{equation}
\label{eq:3.5.10}
\Psi = [\kappa_1, \kappa_2; \kappa_3, \kappa_4] = 
(\kappa_1 - \kappa_2)(\kappa_4 - \kappa_3)/(\kappa_1 - \kappa_3)(\kappa_4 - \kappa_2).
\end{equation}
The ordering of the principal curvatures implies that $\Psi$ must satisfy $0 < \Psi < 1$.  Using equations
(\ref{eq:3.5.9}) and (\ref{eq:3.5.10}), one can compute that 
$\Psi = 1/2$ on any isoparametric hypersurface,
i.e., the four curvature spheres form a 
{\em harmonic set} in the sense of projective geometry (see, for example,
\cite[p. 59]{Sam}).  

There is, however, a simpler way to compute $\Psi$.  One applies Theorem~\ref{thm:3.5.2}
to the Legendre lift of
one of the focal submanifolds of the isoparametric hypersurface.  By the work of M\"{u}nzner,
each isoparametric hypersurface $M^{n-1}$ embedded in $S^n$ has two distinct 
focal submanifolds, each of
codimension greater than one.  The hypersurface $M^{n-1}$ is a 
tube of constant radius over each of these focal
submanifolds.  Therefore, the Legendre lift of $M^{n-1}$ is obtained from the Legendre lift of
either focal
submanifold by parallel transformation.  Thus, the Legendre lift of
$M^{n-1}$ has the same Lie curvature as the 
Legendre lift of either focal submanifold.  Let $\phi:V \rightarrow S^n$ be one of the focal submanifolds.
By the same calculation that yields equation (\ref{eq:3.5.8}), M\"{u}nzner showed that if $\xi$ is any unit normal
to $\phi(V)$ at any point, then the shape operator $A_\xi$ has three distinct principal curvatures,
\begin{displaymath}
\kappa_1 = -1, \quad  \kappa_2 = 0, \quad \kappa_3 = 1.  
\end{displaymath}
By Theorem~\ref{thm:3.5.2}, the Legendre lift of $\phi$ has
a fourth principal curvature $\kappa_4 = \infty$.  Thus, the Lie curvature of this Legendre submanifold is
\begin{displaymath}
\Psi = (-1 - 0)(\infty - 1)/(-1 -1)(\infty - 0) = 1/2.
\end{displaymath}

In the case $g=4$, one can ask what is the strength of the assumption $\Psi = 1/2$ on $M^{n-1}$.  Since $\Psi$ is
only one function of the principal curvatures, one would not expect this assumption to classify Legendre
submanifolds up to Lie equivalence.  However, if one makes additional assumptions, e.g., the Dupin condition,
then results can be obtained.  

Miyaoka \cite{Mi3} proved that the assumption that $\Psi$ is 
constant on a 
compact connected proper Dupin hypersurface 
$M^{n-1}$ in $S^n$ with four principal curvatures, 
together with an additional assumption regarding intersections
of leaves of the various principal foliations, implies that $M^{n-1}$ is Lie equivalent to an isoparametric
hypersurface.  This proved that Conjecture \ref{conj:CCJ-compact} is true under Miyaoka's additional assumptions.
However, no one has proved that Miyaoka's additional assumptions always hold.
In the same paper, Miyaoka also proved that if the Lie curvature of a compact proper Dupin hypersurface with $g=4$ 
is constant, then it has the value $1/2$. 

As mentioned above, Thorbergsson 
\cite{Th1} showed that for a compact proper Dupin hypersurface in $S^n$ with four principal curvatures, the multiplicities 
of the principal curvatures must satisfy $m_1 = m_3, m_2 = m_4$, when the principal
curvatures are ordered as in equation (\ref{eq:3.5.9a}),
the same restriction as for an isoparametric hypersurface (see, for example, \cite[p.108]{CR8}).

Cecil, Chi and Jensen \cite{CCJ2} then used a different approach than Miyaoka \cite{Mi3} to try to prove 
Conjecture \ref{conj:CCJ-compact}.  First, they proved 
that the Legendre lift of a connected, compact proper Dupin hypersurface with $g>2$ principal curvatures is irreducible
(in the sense of Pinkall \cite{P4}, see Section \ref{sec:4.1}).  Then
they worked locally using the method of moving frames and assuming irreducibility to prove the following theorem.

\begin{theorem}
\label{thm:CCJ}
Let $M$ be an irreducible connected proper Dupin 
hypersurface in $S^n$ with four principal curvatures having multiplicities $m_1 = m_3 \geq 1$, $m_2 = m_4 =1$,
and constant Lie curvature $\Psi = 1/2$. Then $M$ is Lie equivalent to an open
subset of an isoparametric hypersurface.
\end{theorem}

Theorem \ref{thm:CCJ} implies the following result in the compact case due to 
Cecil, Chi and Jensen \cite{CCJ2}:

\begin{theorem}
\label{thm:CCJ-compact}
Let $M$ be a compact, connected proper Dupin hypersurface in $S^n$ with four principal curvatures
having multiplicities $m_1 = m_3 \geq 1$, $m_2 = m_4 =1$, and constant Lie curvature.  Then $M$ is Lie equivalent
to an isoparametric hypersurface.\index{Cecil-Chi-Jensen} 
\end{theorem}  

Theorem \ref{thm:CCJ} implies Theorem \ref{thm:CCJ-compact} for the following reasons.  As noted above,
Cecil, Chi and Jensen  \cite{CCJ2} proved that the Legendre lift of a connected, compact proper Dupin hypersurface with $g>2$ principal curvatures is irreducible.  Next Thorbergsson's \cite{Th1} 
result proves that the multiplicities of a compact, connected proper Dupin hypersurface with four principal curvatures satisfy the conditions $m_1 = m_3$ and  $m_2 = m_4$, when the principal curvatures are appropriately ordered.
Finally, Miyaoka \cite{Mi3} proved that if the Lie curvature of a compact proper Dupin hypersurface with $g=4$ 
is constant, then it has the value $1/2$. Thus, the hypotheses of Theorem \ref{thm:CCJ-compact} imply
that the hypotheses of Theorem \ref{thm:CCJ} are satisfied.

This means that the full Conjecture \ref{conj:CCJ-compact} for $g=4$ would be proven, if 
the assumption that the value of $m_2 = m_4$ is equal to one could be eliminated from 
Theorem \ref{thm:CCJ}, and thus from Theorem \ref{thm:CCJ-compact}.  

The proof of Theorem \ref{thm:CCJ} involves some complicated calculations,
which would become even more elaborate if the value of  $m_2 = m_4$ were allowed to be greater than one.
Even so, this approach to proving Conjecture \ref{conj:CCJ-compact} 
could possibly be successful with some additional insight regarding the structure of the calculations involved.
The criteria for reducibility (Theorem \ref{thm:4.2.9}) and the criteria for Lie equivalence 
to an isoparametric hypersurface (Theorem \ref{thm:3.5.6}) 
are both important in the proof of Theorem \ref{thm:CCJ}.

In the case $g=6$, we do not know of any results beyond those of Miyaoka \cite{Mi4}, who
showed that if the Dupin hypersurface
satisfies some additional assumptions
on the intersections of the leaves of the various principal foliations, then
Conjecture \ref{conj:CCJ-compact} is true for $g=6$.
However, no one has proved that Miyaoka's additional assumptions always hold.

A local approach assuming irreducibility,
similar to that used by Cecil, Chi and Jensen for the $g=4$ case, might possibly work
for the case $g=6$, but the calculations involved
would be very complicated, unless some new algebraic insight is found to simplify the situation.

\subsection{Examples with constant Lie curvature that are not Lie equivalent to an isoparametric hypersurface}
\label{sec:3.8}

The following example  \cite{Cec4} (see also \cite[pp. 80--82]{Cec1}) 
is a noncompact proper Dupin hypersurface with $g=4$ distinct principal 
curvatures of multiplicity one, and constant Lie curvature
$\Psi = 1/2$, which is not Lie equivalent to an open subset of an isoparametric
hypersurface with four principal curvatures in $S^n$.  

This example 
is reducible in the sense of Pinkall (see Section \ref{sec:4.1}), 
and it cannot be made compact without destroying the property
that the number $g$ of distinct curvatures spheres equals four at each point.  
This example
shows that some additional hypotheses (either compactness or
irreducibility)
besides constant Lie curvature $\Psi = 1/2$ are needed
to conclude that a proper Dupin hypersurface 
with $g=4$ principal curvatures is Lie equivalent to
an isoparametric hypersurface.  

\begin{example}
\label{ex:constant -Lie-curv}
{\em Let $\phi:V \rightarrow S^{n-m}$ be
an embedded Dupin hypersurface in $S^{n-m}$ with field of unit normals $\xi$ in $S^{n-m}$, such that $\phi$ has three distinct
principal curvatures, 
\begin{displaymath}
\mu_1 <\mu_2 < \mu_3,
\end{displaymath}
at each point of $V$.  Embed $S^{n-m}$ as a totally geodesic submanifold
of $S^n$, and let $B^{n-1}$ be the unit normal
bundle of the submanifold $\phi(V)$ in $S^n$.  Let
$\lambda: B^{n-1} \rightarrow \Lambda^{2n-1}$ be the Legendre lift of
the submanifold $\phi(V)$ in $S^n$.  Any unit normal
$\eta$ to $\phi(V)$ at a point $x \in V$ can be written in the form
\begin{displaymath}
\eta = \cos \theta \ \xi (x) + \sin \theta \ \zeta,
\end{displaymath}
where $\zeta$ is a unit normal to $S^{n-m}$ in $S^n$.  Since the shape operator $A_\zeta = 0$, we have
\begin{displaymath}
A_\eta = \cos \theta \ A_\xi.
\end{displaymath}
Thus the principal curvatures of $A_\eta$ are
\begin{equation}
\label{eq:3.5.11}
\kappa_i = \cos \theta \ \mu_i, \quad 1 \leq i \leq 3.
\end{equation}
If $\eta \cdot \xi = \cos \theta \neq 0$, then $A_\eta$ has three distinct principal curvatures.  However, if
$\eta \cdot \xi = 0$, then  $A_\eta = 0$.  

Let $U$ be the open subset of $B^{n-1}$ on which $\cos \theta > 0$,
and let $\alpha$ denote the restriction of $\lambda$ to $U$.  By Theorem~\ref{thm:3.5.2}, $\alpha$ has four
distinct curvature spheres at each point of $U$.  Since $\phi(V)$ is Dupin in $S^{n-m}$, it is easy to show that
$\alpha$ is Dupin (see the tube construction \cite[pp. 127-133]{Cec1} for the details).  
Furthermore, since $\kappa_4 = \infty$,
the Lie curvature $\Psi$ of $\alpha$ at a point $(x, \eta)$ of $U$ equals the M\"{o}bius curvature 
$\Phi(\kappa_1,\kappa_2,\kappa_3)$.  Using equation (\ref{eq:3.5.11}), we compute
\begin{equation}
\label{eq:3.5.12}
\Psi = \Phi(\kappa_1,\kappa_2,\kappa_3) = \frac{\kappa_1 - \kappa_2}{\kappa_1 - \kappa_3} =
\frac{\mu_1 - \mu_2}{\mu_1 - \mu_3} = \Phi(\mu_1,\mu_2,\mu_3).
\end{equation}

Now consider the special case where
 $\phi(V)$ is a minimal isoparametric hypersurface in $S^{n-m}$ with three distinct principal
curvatures.  By M\"{u}nzner's formula (\ref{eq:3.5.8}), these principal curvatures must have the values,
\begin{displaymath}
\mu_1 = - \sqrt{3}, \quad \mu_2 = 0, \quad \mu_3 = \sqrt{3}.
\end{displaymath}
On the open subset $U$ of $B^{n-1}$ described above, the Lie 
curvature of $\alpha$ has the constant value $1/2$ by equation (\ref{eq:3.5.12}).  To construct an immersed proper Dupin
hypersurface 
with four principal curvatures and constant Lie curvature $\Psi = 1/2$ in $S^n$, we simply take the open
subset $P_{-t} \alpha (U)$ of the tube of radius $t$ around $\phi(V)$ in $S^n$ for an appropriate value of $t$.

We can see that this example is not Lie equivalent to an open subset of an isoparametric hypersurface
in $S^n$ with four distinct principal curvatures as follows.  Note that
the example that we have constructed is reducible, because 
it is obtained by the tube construction from the hypersurface $\phi(V)$ in $S^{n-m} \subset S^n$.  
On the other hand, as we will show in the next paragraph, any open subset of an isoparametric hypersurface with four principal curvatures in $S^n$ 
is irreducible, so it cannot be Lie equivalent to our example.

The irreducibility of an open subset of an isoparametric hypersurface with four principal curvatures in $S^n$ 
follows from the results of Cecil, Chi and Jensen \cite{CCJ2}, who proved that
a compact proper Dupin hypersurface with $g>2$ principal curvatures is irreducible.  
Furthermore, they proved that 
an irreducible proper Dupin hypersurface cannot contain a reducible open subset.
This follows from the algebraicity of proper Dupin hypersurfaces, as stated in
Theorem \ref{algebraic} (see also \cite[pp. 145-147]{Cec1}).  Thus, a compact isoparametric 
hypersurface with four principal curvatures in $S^n$ is irreducible, and it cannot contain a reducible open subset.

Note that the 
point sphere map $[Y_1]$ of  the Legendre submanifold $\lambda: B^{n-1} \rightarrow \Lambda^{2n-1}$ defined above
is a curvature sphere of multiplicity $m$ which lies in the linear subspace of codimension $m+1$ in
${\bf P}^{n+2}$ orthogonal to the space spanned by $e_{n+3}$ and by those vectors $\zeta$ normal to $S^{n-m}$ in $S^n$.
This geometric fact implies that for such a vector $\zeta$, there are only two distinct curvature spheres on each
of the lines $\lambda (x, \zeta)$, since $A_\zeta = 0$ (see Theorem~\ref{thm:3.5.2}).  
This change in the number
of distinct curvature spheres at points of the form $(x, \zeta)$ is precisely why $\alpha$ cannot be extended to a
compact proper Dupin submanifold with $g=4$.}
\end{example}

With regard to Theorem~\ref{thm:3.5.6}, $\alpha$ comes as close as possible to satisfying the requirements
for being Lie equivalent
to an isoparametric hypersurface without actually fulfilling them.  The principal
curvatures $\kappa_2 = 0$ and $\kappa_4 = \infty$ are constant on $U$.  If a third principal curvature
were also constant, then the constancy of $\Psi$ would imply that all four principal curvatures were constant,
and $\alpha$ would be the Legendre lift of an isoparametric hypersurface.

Using this same method, it is easy to construct reducible noncompact proper 
Dupin hypersurfaces in $S^n$ with $g=4$
and $\Psi = c$, for any constant $0 < c < 1$.  If $\phi(V)$ is an isoparametric hypersurface in $S^{n-m}$ with
three distinct principal curvatures, then M\"{u}nzner's formula (\ref{eq:3.5.8}) implies that these principal
curvatures must have the values,
\begin{equation}
\label{eq:3.5.13}
\mu_1 = \cot (\theta + \frac{2 \pi}{3}), \quad \mu_2 = \cot (\theta + \frac{\pi}{3}), \quad
\mu_3 = \cot \theta, \quad 0 < \theta < \frac{\pi}{3}. 
\end{equation}
Furthermore, any value of $\theta$ in $(0, \pi/3)$ can be realized by some hypersurface in a parallel family
of isoparametric hypersurfaces in $S^{n-m}$.  
A direct calculation using equations (\ref{eq:3.5.12}) and (\ref{eq:3.5.13}) shows
that the Lie curvature $\Psi$ of $\alpha$ satisfies
\begin{displaymath}
\Psi = \Phi(\kappa_1,\kappa_2,\kappa_3) = \frac{\kappa_1 - \kappa_2}{\kappa_1 - \kappa_3} =
\frac{\mu_1 - \mu_2}{\mu_1 - \mu_3}  
= \frac{1}{2} + \frac{\sqrt{3}}{2} \tan (\theta - \frac{\pi}{6})
\end{displaymath}
on $U$.  This can assume any value $c$ in the interval $(0,1)$ by an appropriate choice of $\theta$ in
$(0, \pi/3)$.  An open subset of a tube over $\phi(V)$ in $S^n$ is a proper Dupin hypersurface with $g=4$
and $\Psi = \Phi = c$.  Note that $\Phi$ has different values on different hypersurfaces in the parallel
family of isoparametric hypersurfaces in $S^{n-m}$.  Thus these hypersurfaces are not 
M\"{o}bius equivalent to each other by Corollary
\ref{cor:3.5.5}.  This is consistent with the fact that a 
parallel transformation is not a M\"{o}bius transformation.

\subsection{Local classifications of Dupin hypersurfaces}
\label{sec:3.9}

In this section, we mention local classifications of 
proper Dupin hypersurfaces that have been obtained using the methods of Lie sphere geometry.
In the case of $g=2$ principal curvatures, the classification of Dupin surfaces in $S^3$ (the cyclides of Dupin)
goes back to the nineteenth century. (See, \cite[pp. 151--166]{CR7} 
or \cite[p. 148]{Cec1} for more detail.) 

A proper Dupin hypersurface $M \subset S^n$ with two distinct curvature spheres of respective multiplicities $p$
and $q$ is called a {\em cyclide of Dupin of characteristic} $(p,q)$ (see Pinkall \cite{P4}). In the same paper,
Pinkall \cite{P4} proved the following classification of the higher dimensional cyclides
(see also \cite[pp. 149--151]{Cec1} for a proof).

\begin{theorem}
\label{thm:dupin-cyclides}
\begin{enumerate}
\item[${\rm(a)}$] Every connected cyclide of Dupin of characteristic $(p,q)$
is contained in a unique compact, connected cyclide of Dupin of characteristic $(p,q)$.
\item[${\rm(b)}$] Any two cyclides of Dupin of the same characteristic $(p,q)$ are locally Lie equivalent, each being
Lie equivalent to an open subset of a standard product of two spheres
\begin{equation}
\label{eq:std-prod-p-q-2}
S^q (1/\sqrt{2}) \times S^p (1/\sqrt{2}) \subset S^n \subset {\bf R}^{q+1} \times {\bf R}^{p+1} = {\bf R}^{n+1}, 
\end{equation} 
where $p+q = n-1$.
\end{enumerate}
\end{theorem}

From this, one can derive a M\"{o}bius geometric classification of the cyclides of Dupin as follows
(see \cite{Cec3}, \cite[pp. 151--159]{Cec1}).

\begin{theorem}
\label{thm:4.3.2}
${\rm(a)}$ Every connected cyclide of Dupin $M^{n-1} \subset {\bf R}^n$ of characteristic $(p,q)$ is M\"{o}bius
equivalent to an open subset of a hypersurface of revolution\index{surface of revolution} obtained by revolving a $q$-sphere 
$S^q \subset {\bf R}^{q+1} \subset {\bf R}^n$ about an axis ${\bf R}^q \subset {\bf R}^{q+1}$ or a $p$-sphere
$S^p \subset {\bf R}^{p+1} \subset {\bf R}^n$ about an axis ${\bf R}^p \subset {\bf R}^{p+1}$.

\noindent
${\rm(b)}$ Two hypersurfaces obtained by revolving a $q$-sphere 
$S^q \subset {\bf R}^{q+1} \subset {\bf R}^n$ about an axis of revolution ${\bf R}^q \subset {\bf R}^{q+1}$ are M\"{o}bius
equivalent if and only if they have the same value of
$\rho = |r|/a$, where $r$ is the signed radius
of the profile sphere $S^q$  and $a>0$ is the distance from the center of $S^q$ to the axis of revolution.\\
\end{theorem}

 In the case $g=3$, Pinkall \cite{P1}, \cite{P3} (see also the English translation  \cite{P1}),  obtained a
classification up to Lie equivalence of proper Dupin hypersurfaces
 in ${\bf R}^4$ with three distinct principal curvatures of multiplicity one.
 Another version of Pinkall's proof in the irreducible case was given later by Cecil
and Chern \cite{CC2} (see also \cite{CC2u}, \cite[pp. 168--190]{Cec1}).  
Cecil and Chern developed the method of moving Lie frames which can be applied to
the general study of Legendre submanifolds.  This approach has been applied successfully by 
Niebergall   \cite{N1}--\cite{N2},
Cecil and Jensen \cite{CJ2}--\cite{CJ3}, 
and by Cecil, Chi and Jensen \cite{CCJ2} to obtain local classifications
of higher-dimensional Dupin hypersurfaces.  

In particular, Cecil and Jensen \cite{CJ2} proved that for a connected irreducible proper
Dupin hypersurface
with three principal curvatures, all of the multiplicities must be equal, and the hypersurface
must be Lie equivalent to an open subset of an isoparametric hypersurface in 
a sphere with three principal curvatures.  

An open problem is the classification of reducible Dupin hypersurfaces of arbitrary dimension with $g = 3$
principal curvatures up to Lie equivalence.  Pinkall \cite{P1}, \cite{P3}, found such a classification in the case 
$M^3 \subset {\bf R}^4$, and it may be possible to obtain results in higher dimensions using his approach.

As noted in Section \ref{sec:3.7}, local classification in the case $g=4$ has been studied extensively
in the papers of Niebergall \cite{N1}--\cite{N2},
Cecil and Jensen \cite{CJ2}--\cite{CJ3},
and by Cecil, Chi and Jensen\index{Cecil--Chi--Jensen} \cite{CCJ2}, 
culminating in Theorem \ref{thm:CCJ} of Section \ref{sec:3.7}.  The method of moving
frames is central to the approach of those papers.

As to local classifications in the case $g=6$, an approach similar to that of Cecil, Chi and Jensen \cite{CCJ2} is plausible, but the calculations involved would be
very complicated, unless some new algebraic insight is found to simplify the situation.

\noindent Department of Mathematics and Computer Science

\noindent College of the Holy Cross

\noindent Worcester, MA 01610, U.S.A.

\noindent email: tcecil@holycross.edu

\end{document}